\documentclass[a4paper,11pt,reqno]{amsart}

\usepackage{amssymb}
\usepackage[all]{xy}
\usepackage{hyperref}
\usepackage{xcolor}
\usepackage{enumitem}
\usepackage{subcaption}
\usepackage[final]{graphicx}
\usepackage{float}
\usepackage{epic}
\usepackage{setspace}
\usepackage{color}
\usepackage{verbatim}

\numberwithin{equation}{section}
\newtheorem{theorem}{Theorem}[section]

\newtheorem{lemma}[theorem]{Lemma}

\newtheorem{proposition}[theorem]{Proposition}

\newtheorem{corollary}[theorem]{Corollary}
\newtheorem*{theorem*}{Theorem}
\newtheorem*{question*}{Question}
\newtheorem*{corollary*}{Corollary}
\newtheorem*{definition*}{Definition}

\newtheorem{letterthm}{Theorem}

\theoremstyle{definition}
\newtheorem{definition}[theorem]{Definition}

\theoremstyle{remark}
\newtheorem{example}[theorem]{Example}

\theoremstyle{remark}

\newtheorem{remark}[theorem]{Remark}

\begin{document}

\title{Complementation in continuous cohomology with coefficients in Banach modules}

\author{Mario Klisse}

\address{KU Leuven, Department of Mathematics,
Celestijnenlaan 200B, 3001 Leuven, Belgium}

\email{mario.klisse@kuleuven.be}

\begin{abstract}

In this article, we introduce the concept of weakly uniquely stationary representations. This framework enables us to investigate the complementability of closed subspaces within the context of continuous cohomology with coefficients in Banach modules. As an application, we extend and refine several cohomological results from the literature, particularly in the settings of nilpotent groups, products of groups, and lattices.

\end{abstract}

\date{\today. \emph{MSC2010:}  20J06, 22D12, 46B99. The author is supported by the FWO postdoctoral grant 1203924N of the Research Foundation Flanders.}

\maketitle

\section*{Introduction}

The study of group cohomology with coefficients in Banach modules is a rich area of mathematics that lies at the intersection of several fields, including algebra, topology, and functional analysis. Its motivations are both theoretical, in terms of understanding the algebraic and topological properties of groups and their actions, and practical, in terms of applications to areas like representation theory, harmonic analysis, and geometric group theory. In the context of unitary representations on Hilbert spaces, a notable application of the theory is the characterization of Kazhdan's Property (T) through the vanishing of 1-cohomology groups (see \cite{Delorme77}).

While the continuous cohomology of groups with coefficients in unitary representations on Hilbert spaces is well-studied, attempts of extending existing results to more general Banach modules often lead to technical difficulties due to the lack of concepts like orthogonality. One example of this phenomenon is the complementability of closed subspaces, whose study for general Banach modules was initiated by Nowak in \cite{Nowak17}. Recall that a subspace $W$ of a topological vector space $V$ is \emph{complemented}, if there exists a subspace $W^{\prime}\subseteq V$ such that $V$ decomposes as a topological direct sum $V=W\oplus W^{\prime}$. Under suitable conditions, particularly covering all isometric representations without almost invariant vectors on uniformly convex Banach spaces, the main result of \cite{Nowak17} proves the complementation of the subspace of 1-coboundaries in the space of all 1-cocycles. This result was later extended to higher degrees by Rosendal in a more restrictive setting (see \cite{Rosendal22}).\\

In this paper, we investigate the complementation of closed subspaces within the context of continuous cohomology with coefficients in Banach modules. To this end, we introduce the concept of weakly uniquely stationary representations.

\begin{definition*} \emph{Let $G$ be a topological group and $\mu$ a Baire probability measure on $G$. We call a uniformly bounded linear representation $\rho$ of $G$ on a Banach space $V$ \emph{weakly uniquely $\mu$-stationary} if it is $\mu$-integrable (in the sense of Definition \ref{PettisIntegrableDefinition}) and if every non-trivial functional that is invariant under the convolution with the probability measure $\mu$ admits an element $v\in V$ with $\phi(v)\neq0$ that is invariant under the convolution $\rho_{\mu}$ with $\mu$.} \end{definition*}

The $\rho_{\mu}$-invariant elements $V^{\mu}$ of a weakly uniquely $\mu$-stationary Banach $G$-module $V$ form the range of a unique projection $\widetilde{\mathbb{E}}_{\mu}$ on $V$. Our definition draws inspiration from the concept of uniquely stationary actions of groups on C$^{\ast}$-algebras, as introduced by Hartman and Kalantar in \cite{HartmanKalantar}. It unifies a number of existing concepts, such as weakly almost periodic Banach modules (see \cite{BaderRosendalSauer} and also \cite{CornulierTessera20}), strictly convex Banach modules, and representations vanishing at infinity. Additionally, it naturally relates to the work of Jaworski and Neufang on the Choquet--Deny equation in dual Banach spaces in \cite{JaworskiNeufang}, and to results of Bekka and Valette in \cite{BekkaValette97}.

Motivated by an important ingredient in the proof of Shalom's cohomological rigidity results for unitary representations of locally compact groups in \cite{Shalom00}, in \cite{BaderRosendalSauer} Bader, Rosendal and Sauer proved the vanishing of certain restriction maps in continuous reduced cohomology of weakly almost periodic Banach modules. We extend their main theorem to the setting of weakly uniquely stationary representations, enabling us to derive complementation results for the corresponding higher-order cohomology groups.

\begin{letterthm} Let $G$ be a topological group containing commuting subgroups $N,C\leq G$. Let furthermore $\mu$ be a Baire probability measure on $C$, $(V,\rho)$ a uniformly bounded, strongly operator continuous, weakly uniquely $\mu$-stationary Banach $G$-module, and $b\in Z^{n}(G,\rho)$ a $n$-cocycle. Then $\widetilde{\mathbb{E}}_{\mu}\circ(b|_{N^{n+1}})$ defines a $n$-cocycle that is almost cohomologous to the restriction $b|_{N^{n+1}}$. \end{letterthm}

The study of the complementation of closed subspaces in continuous cohomology has several intriguing applications, such as in Hodge theory, in the deduction of the vanishing of certain cohomology groups (see e.g. \cite{Garland73}), and in Ozawa's functional analytic proof of Gromov's polynomial growth theorem presented in \cite{Ozawa18}. Additionally, it has been applied to address rigidity questions concerning unitary representations of irreducible integrable lattices in products of groups as in \cite{Shalom00}, \cite{BaderFurmanGelanderMonod} (see also \cite{Creutz11}, \cite{Creutz22}). Inspired by this, we utilize our aforementioned results to examine the decomposition of the continuous cohomology of products of groups. 

\begin{letterthm} Let $G:=G_{1}\times G_{2}$ be a product of topological groups and $(V,\rho)$ a uniformly bounded strongly operator continuous Banach $G$-module. Assume that $G_{1}$ and $G_{2}$ admit Baire probability measures $\mu_{1}$ and $\mu_{2}$ and null sets $\mathcal{N}_{1}\subseteq G_{1}$, $\mathcal{N}_{2}\subseteq G_{2}$ for which both $G_{1}\setminus\mathcal{N}_{1}$ and $G_{2}\setminus\mathcal{N}_{2}$ are contained in compact sets. Assume furthermore that $\Vert\rho_{\mu_{1}}\rho_{\mu_{2}}\Vert<1$ and that $V^{\mu_{1}}=V^{G_{1}}$ and $V^{\mu_{2}}=V^{G_{2}}$. Then there exists a topological isomorphism of continuous cohomology groups $H_{c}^{1}(G_{1}\times G_{2},V)\cong H_{c}^{1}(G_{1},V^{G_{2}})\oplus H_{c}^{1}(G_{2},V^{G_{1}})$. \end{letterthm}

\begin{letterthm} Let $G:=G_{1}\times G_{2}$ be a product of topological groups, $\mu_{1}$, $\mu_{2}$ Baire probability measures on $G_{1}$ and $G_{2}$, and $(V,\rho)$ a uniformly bounded strongly operator continuous Banach $G$-module. Assume that $(V,\rho|_{G_{1}})$ is weakly uniquely $\mu_{1}$-stationary without almost invariant unit vectors, and that $(V,\rho|_{G_{2}})$ is weakly uniquely $\mu_{2}$-stationary. Then there exists an embedding of topological vector spaces $H_{c}^{1}(G_{1}\times G_{2},V)\hookrightarrow H_{c}^{1}(G_{1},V^{\mu_{2}})\oplus H_{c}^{1}(G_{2},V^{\mu_{1}})$. If furthermore $V^{\mu_{1}}=V^{G_{1}}$ and $V^{\mu_{2}}=V^{G_{2}}$, then the embedding is surjective, i.e. $H_{c}^{1}(G_{1}\times G_{2},V)\cong H_{c}^{1}(G_{1},V^{G_{2}})$. \end{letterthm}

Building on the ideas presented in \cite{Shalom00}, \cite{BaderFurmanGelanderMonod}, \cite{Creutz11}, we use these decompositions to derive new superrigidity results for suitable representations of irreducible integrable lattices in products of groups.\\

\noindent \emph{Structure}. The article is organized as follows. In Section \ref{sec:Preliminaries-and-notation}, we review fundamental concepts related to the theory of Baire sets and Baire measures, Banach modules, and group cohomology. In Section \ref{StationaryElements} we investigate elements in the bidual of Banach spaces that are stationary with respect to convolution with a given Baire probability measure. Motivated by our results, in Section \ref{sec:The-projection} we introduce the concept of weakly uniquely stationary representations and explore various examples, including uniquely stationary actions of discrete groups on C$^{\ast}$-algebras, weakly almost periodic Banach modules, and linear representations vanishing at infinity. Finally, in Section \ref{sec:Application-to-cohomology}, we apply the notion of weakly uniquely stationary representations to the study of group cohomology with coefficients in Banach modules, obtaining complementation and rigidity results.

\vspace{3mm}



\section{Preliminaries and notation\label{sec:Preliminaries-and-notation}}

\subsection{General notation}

In this article, all Banach spaces will be assumed to be real or complex. Given a Banach space $V$, we denote the bounded linear operators on $V$ by $\mathcal{B}(V)$ and $\text{GL}(V)\subseteq\mathcal{B}(V)$ is the subset of invertible bounded linear operators on $V$. The dual space and bidual space of $V$ will be denoted by $V^{\ast}$ and $V^{\ast\ast}$ respectively. We furthermore write $C_{b}(X,\mathbb{R})$ for the bounded continuous real-valued functions on a topological space $X$. The topological groups in this article are not necessarily assumed to be Hausdorff.

\vspace{3mm}


\subsection{Baire sets and Baire measures}

In the study of measures in terms of integrals of continuous functions, it is often convenient or even required to pass to $\sigma$-algebras that are smaller than the corresponding Borel $\sigma$-algebra. In the general topological setting considered in this article, the notion of Baire measures will be suitable. Note that there exist several (inequivalent) definitions of this concept; we will be following the conventions and notation of \cite{Bogachev} and \cite{Fremlin}.

Define the \emph{Baire $\sigma$-algebra} $\mathcal{B}a(X)$ of a topological space $X$ as the smallest $\sigma$-algebra with respect to which all continuous real-valued functions on $X$ are measurable. In general $\mathcal{B}a(X)$ does not need to coincide with the Borel $\sigma$-algebra of $X$; for metrizable spaces however this is the case. A (signed) measure on $\mathcal{B}a(X)$ is called a \emph{Baire measure} and the set of all Baire measures will be denoted by $\mathcal{M}_{\sigma}(X)$. Its subset of finite Baire measures can be equipped with the (locally convex, Hausdorff) \emph{weak topology} induced by the requirement that a net $(\mu_{\lambda})_{\lambda\in\Lambda}$ \emph{weakly converges} to $\mu$ if and only if $\int_{X}fd\mu_{\lambda}\rightarrow\int_{X}fd\mu$ for all $f\in C_{b}(X,\mathbb{R})$. (Note that this terminology, despite being standard, conflicts with the standard one in functional analysis.) We denote the subspace of \emph{Baire probability measures} on $X$ by $\mathcal{P}_{\sigma}(X)$ and equip it with the subspace topology.\\

We will make use of the following well-known statement, see e.g. \cite[Example 4.2.5]{Bogachev}. We include its proof for the convenience of the reader.

\begin{lemma} \label{MeasureDensity} Let $X$ be a topological space. Then the convex hull of the set of all point measures in $X$ is weakly dense in $\mathcal{P}_{\sigma}(X)$. \end{lemma}

\begin{proof} Let $\mu\in\mathcal{P}_{\sigma}(X)$. Sets of the form 
\[
\mathcal{U}_{f_{1},...,f_{n},\varepsilon}(\mu):=\left\{ \nu\in\mathcal{P}_{\sigma}(X)\mid\left|\int_{X}f_{i}d\mu-\int_{X}f_{i}d\nu\right|<\varepsilon\text{ for }i=1,...,n\right\} 
\]
with $\varepsilon>0$ and $f_{1},...,f_{n}\in C_{b}(X,\mathbb{R})$ define a neighbourhoods basis of $\mu$. It hence suffices to show that every such set intersects non-trivially with the convex hull of all point measures $\delta_{x}$, $x\in X$. Since the functions $f_{1},...,f_{n}$ are measurable and bounded, we find simple functions $g_{1},...,g_{n}$ for which $\sup_{1\leq i\leq n}\Vert f_{i}-g_{i}\Vert_{\infty}<\varepsilon/2$. The $g_{i}$ are (without loss of generality) of the form $g_{i}=\sum_{j=1}^{m}c_{i,j}\chi_{M_{j}}$ with pairwise disjoint Baire sets $M_{1},...,M_{m}\in\mathcal{B}a(X)$ and $c_{i,1},...,c_{i,m}\in\mathbb{R}$. For every $1\leq j\leq m$ pick $x_{j}\in M_{j}$, let $y\in X\setminus(\bigcup_{j=1}^{m}M_{j})$ and define $\nu:=\mu\left(X\setminus(\bigcup_{j=1}^{m}M_{j})\right)\delta_{y}+\sum_{j=1}^{m}\mu(M_{j})\delta_{x_{j}}$. Then $\nu$ is a Baire probability measure with $\int_{X}g_{i}d\mu=\sum_{j=1}^{m}c_{i,j}\mu(M_{j})=\int_{X}g_{i}d\nu$ for $i=1,...,n$ and hence $\nu\in\mathcal{U}_{f_{1},...,f_{n},\varepsilon}(\mu).$ \end{proof}

\vspace{3mm}


\subsection{Banach modules}

Let $G$ be a topological group and $\rho:G\rightarrow\text{GL}(V)$, $g\mapsto\rho_{g}$ a uniformly bounded linear representation of $G$ on a Banach space $V$. For the sake of brevity we also say that $(V,\rho)$ is a \emph{uniformly bounded Banach $G$-module} and simply write $g.v:=\rho_{g}(v)$ for $g\in G$, $v\in V$. The representation $\rho$ naturally induces a uniformly bounded linear representation $\rho^{\ast}:G\rightarrow\text{GL}(V^{\ast})$ of $G$ on the dual space $V^{\ast}$ via $\rho_{g}^{\ast}(\phi):=\phi\circ\rho_{g^{-1}}$ for $g\in G$, $\phi\in V^{\ast}$, which is called the \emph{dual representation} or \emph{contragradient representation}. Recall that there exists an isometric embedding $\iota_{V}:V\hookrightarrow V^{\ast\ast}$ of $V$ into its bidual via $(\iota_{V}(v))(\phi):=\phi(v)$ for every $v\in V$, $\phi\in V^{\ast}$. We will also write $\widehat{v}:=\iota_{V}(v)$. The representation $\rho$ then extends to a uniformly bounded linear representation $\rho^{\ast\ast}:G\rightarrow\text{GL}(V^{\ast\ast})$ via $\rho_{g}^{\ast\ast}(f):=f\circ\rho_{g^{-1}}^{\ast}$ for $g\in G$, $f\in V^{\ast\ast}$ . With respect to $\rho^{\ast\ast}$ the embedding $\iota_{V}$ is $G$-equivariant and we have that $\sup_{g\in G}\Vert\rho_{g}\Vert=\sup_{g\in G}\Vert\rho_{g}^{\ast}\Vert=\sup_{g\in G}\Vert\rho_{g}^{\ast\ast}\Vert$. By Goldstine's theorem the image of the closed unit ball in $V$ under $\iota_{V}$ is weak$^{\ast}$-dense in $V^{\ast\ast}$, so in particular $\text{im}(\iota_{V})\subseteq V^{\ast\ast}$ is weak$^{\ast}$-dense. If $G$ is a topological group, $\rho$ (or $(V,\rho)$) is called \emph{strongly operator continuous} if the map $G\ni g\mapsto\rho_{g}(v)$ is continuous with respect to the norm on $V$ for every $v\in V$. Similarly, $\rho$ (or $(V,\rho)$) is called \emph{weakly operator continuous} if the map $G\ni g\mapsto\rho_{g}(v)$ is continuous with respect to the weak topology on $V$ for every $v\in V$.

We denote the space of bounded operators on $V$, $V^{\ast}$ and $V^{\ast\ast}$ by $\mathcal{B}(V)$, $\mathcal{B}(V^{\ast})$ and $\mathcal{B}(V^{\ast\ast})$ respectively and endow these spaces with the corresponding operator norms. We can furthermore consider the \emph{point-weak$^{\ast}$ topology} on $\mathcal{B}(V^{\ast\ast})$ which is the locally convex Hausdorff topology induced by the family $\{p_{\phi,f}\}_{\phi\in V^{\ast},f\in V^{\ast\ast}}$ of seminorms with $p_{\phi,f}(T):=\vert\left(T(f)\right)(\phi)\vert$, i.e. $T_{\lambda}\rightarrow T$ for a net $(T_{\lambda})_{\lambda\in\Lambda}\subseteq\mathcal{B}(V^{\ast\ast})$ and $T\in\mathcal{B}(V^{\ast\ast})$ if and only if $\left(T_{\lambda}(f)\right)(\phi)\rightarrow\left(T(f)\right)(\phi)$ for every $f\in V^{\ast\ast}$, $\phi\in V^{\ast}$.

\vspace{3mm}


\subsection{Cohomology of groups}

Let $(V,\rho)$ be a uniformly bounded Banach $G$-module and assume that $\rho$ is strongly operator continuous. For $n\in\mathbb{N}$ denote the vector space of all continuous functions $f:G^{n}\rightarrow V$ by $C^{n}(G,V)$. Here $G^n$ is equipped with the product topology and $V$ with the norm topology. We endow $C^{n}(G,V)$ with the \emph{compact-open topology}; that is, the topology generated by the subbase of sets of the form $\{f\in C^{n}(G,V)\mid f(K)\subseteq U\}$, where $K\subseteq G^{n}$ is compact and $U\subseteq V$ is open. Note that in our particular setting, the compact-open topology coincides with the \emph{topology of uniform convergence on compact subsets}, see e.g. \cite[Theorem 9.21]{Stroppel}. The subspace $C^{n}(G,V)^{G}$ of elements $f\in C^{n}(G,V)$ that are \emph{$G$-equivariant} in the sense that $f(gg_{1},...,gg_{n})=\rho_{g}(f(g_{1},...,g_{n}))$ for all $g,g_{1},...,g_{n}\in G$ is closed in $C^{n}(G,V)$.

Consider the cochain complex 
\[
C^{0}(G,V)^{G}\overset{\partial^{0}}{\longrightarrow}C^{1}(G,V)^{G}\overset{\partial^{1}}{\longrightarrow}...\overset{\partial^{n-1}}{\longrightarrow}C^{n}(G,V)^{G}\overset{\partial^{n}}{\longrightarrow}...,
\]
where $\left(\partial^{n}f\right)(g_{1},...,g_{n+1}):=\sum_{i=1}^{n+1}(-1)^{i+1}f(g_{1},...,\widehat{g}_{i},...,g_{n+1})$, $n\in\mathbb{N}$ denote the\emph{ standard differentials}. The $\partial^{n}$ are continuous with respect to the compact-open topology and one checks that $\partial^{n+1}\circ\partial^{n}=0$ for all $n\in\mathbb{N}$. The subspace of \emph{$n$-cocycles} is given by $Z^{n}(G,\rho):=\ker(\partial^{n+1})\subseteq C^{n+1}(G,V)^{G}$, whereas $B^{n}(G,\rho):=\text{im}(\partial^{n})\subseteq C^{n+1}(G,V)^{G}$ denotes the subspace of \emph{coboundaries}. The quotient 
\[
H_{c}^{n}(G,\rho):=Z^{n}(G,\rho)/B^{n}(G,\rho)
\]
is called the \emph{$n$-th continuous cohomology group} of the representation. In general $B^{n}(G,\rho)$ need not be closed in $C^{n+1}(G,V)^{G}$. One thus also considers the \emph{$n$-th continuous reduced cohomology group} 
\[
\overline{H}_{c}^{n}(G,\rho):=Z^{n}(G,\rho)/\overline{B^{n}(G,\rho)},
\]
by replacing $B^{n}(G,\rho)$ by its closure $\overline{B^{n}(G,\rho)}$. If the representation $\rho$ is clear, we will also write $Z^{n}(G,V):=Z^{n}(G,\rho)$, $B^{n}(G,V):=B^{n}(G,\rho)$, $H_{c}^{n}(G,V):=H_{c}^{n}(G,\rho)$ and $\overline{H}_{c}^{n}(G,V):=\overline{H}_{c}^{n}(G,\rho)$ respectively.

Note that for $n=1$ the space $Z^{1}(G,\rho)$ identifies with $\{f\in C(G,V)\mid f(gh)=\rho_{g}(f(h))+f(g)\text{ for all }g,h\in G\}$ via $Z^{1}(G,\rho)\ni f\mapsto f(e,\cdot)$. In this picture, $B^{n}(G,\rho)$ is the subspace of functions of the form $G\ni g\mapsto v-\rho_{g}(v)$ where $v\in V$and $\overline{B^{n}(G,\rho)}$ is its closure with respect to the compact-open topology.

\vspace{3mm}


\section{Stationary elements in the bidual Banach space\label{StationaryElements}}

Let $G$ be a topological group and $\rho:G\rightarrow\text{GL}(V)$ a uniformly bounded linear representation of $G$ on a Banach space $V$. If $\rho$ is weakly operator continuous, the function $g\mapsto\phi(\rho_{g}(v))$ is Baire measurable and $\mu$-integrable for every Baire probability measure $\mu\in\mathcal{P}_{\sigma}(G)$ and $v\in V$, $\phi\in V^{\ast}$. One can hence define an element in $V^{\ast\ast}$ via 
\begin{equation}
V^{\ast}\rightarrow\mathbb{C}, \, \phi\mapsto\int_{G}\phi\left(\rho_{g}(v)\right)d\mu(g),\label{MarkovOperator}
\end{equation}
where the boundedness of the map follows from the closed graph theorem.

\begin{definition} \label{PettisIntegrableDefinition} Let $G$ be a topological group, $\rho:G\rightarrow\text{GL}(V)$ a uniformly bounded linear representation of $G$ on a Banach space $V$ and $\mu\in\mathcal{P}_{\sigma}(G)$ a Baire probability measure. We call $\rho$ \emph{$\mu$-integrable} if $\rho$ is weakly operator continuous and if for all $v\in V$ the element (\ref{MarkovOperator}) is contained in $\text{im}(\iota_{V})\subseteq V^{\ast\ast}$. We denote its preimage in $V$ by $\rho_{\mu}(v)$ and call the map $\rho_{\mu}:v\mapsto\rho_{\mu}(v)$ the \emph{Markov operator} on $V$ (with respect to $\mu$). For convenience we will often abbreviate by saying that $(V,\rho)$ is a \emph{$\mu$-integrable Banach $G$-module.}

We call $\rho$ \emph{Baire-Pettis integrable} (or simply \emph{BP-integrable}) if $\rho$ is $\mu$-integrable for all $\mu\in\mathcal{P}_{\sigma}(G)$. In this case we also say that $(V,\rho)$ is a \emph{BP-integrable Banach $G$-module.} \end{definition}

\begin{remark} \label{BPRemarks} \emph{(i)} Let $G$ be a topological group, $\rho:G\rightarrow\text{GL}(V)$ a weakly operator continuous uniformly bounded linear representation of $G$ on a Banach space $V$ and $\mu\in\mathcal{P}_{\sigma}(G)$. It follows from Pettis' measurability theorem (see \cite[Theorem 1.1]{Pettis}) that $\rho$ is $\mu$-integrable if for any $v\in V$ there exists a $\mu$-null set $\mathcal{N}\subseteq G$ such that $\{\rho_{g}(v)\mid g\in G\setminus\mathcal{N}\}\subseteq V$ is separable. In particular, $(V,\rho)$ will be BP-integrable if $V$ is separable or if $G$ is countable. Another instance where $\rho$ is BP-integrable is if the Banach space $V$ is reflexive.

\emph{(ii)} Given a topological group $G$, a Baire probability measure $\mu\in\mathcal{P}_{\sigma}(G)$, and a uniformly bounded $\mu$-integrable Banach $G$-module $(V,\rho)$, Lemma \ref{MeasureDensity} implies that the operator $\rho_{\mu}\in\mathcal{B}(V)$ is contained in the point-weak closed convex hull of $\{\rho_{g}\mid g\in G\}\subseteq\mathcal{B}(V)$. An application of the Hahn-Banach separation theorem furthermore implies that for every $v\in V$ the element $\rho_{\mu}(v)$ is contained in the norm closed convex hull $\mathcal{C}$ of $G.v:=\{\rho_{g}(v)\mid g\in G\}\subseteq V$. Hence, for every $g\in G$, 
\[
\rho_{\mu}(\rho_{g}(v))\in\overline{\text{conv}}^{\Vert\cdot\Vert}\{\rho_{hg}(v)\mid h \in G\}=\mathcal{C},
\]
so that, by the linearity and norm continuity of the Markov operator, $\mathcal{C}$ must be invariant under $\rho_{\mu}$. We inductively obtain that for finitely many Baire probability measures $\mu_{1},...,\mu_{i}\in\mathcal{P}_{\sigma}(G)$ the element $(\rho_{\mu_{1}}\circ...\circ\rho_{\mu_{i}})(v)$ is contained in $\mathcal{C}$. In particular, $\Vert(\rho_{\mu_{1}}\circ...\circ\rho_{\mu_{i}})(v)\Vert\leq M\Vert v\Vert$ for every $v\in V$, where $M:=\sup_{g\in G}\Vert\rho_{g}\Vert$. \end{remark}

We furthermore introduce bounded linear operators $\rho_{\mu}^{\ast}\in\mathcal{B}(V^{\ast})$ and $\rho_{\mu}^{\ast\ast}\in\mathcal{B}(V^{\ast\ast})$ via $\rho_{\mu}^{\ast}(\phi):=\phi\circ\rho_{\check{\mu}}$ and $\rho_{\mu}^{\ast\ast}(f):=f\circ\rho_{\check{\mu}}^{\ast}$ for $\phi\in V^{\ast}$, $f\in V^{\ast\ast}$, where $\check{\mu}\in\mathcal{P}_{\sigma}(G)$ is the \emph{symmetric opposite} of $\mu$ given by $\check{\mu}(E):=\mu(E^{-1})$ for $E\in\mathcal{B}a(G)$. Note that 
\begin{equation}
\left(\rho_{\mu}^{\ast\ast}(\widehat{v})\right)(\phi)=\widehat{v}\circ\rho_{\check{\mu}}^{\ast}(\phi)=\widehat{v}(\phi\circ\rho_{\mu})=\phi\circ\rho_{\mu}(v)=\widehat{\rho_{\mu}(v)}(\phi)\label{eq:Compatibility}
\end{equation}
for all $v\in V$, $\phi\in V^{\ast}$, i.e. $\rho_{\mu}^{\ast\ast}\circ\iota_{V}=\iota_{V}\circ\rho_{\mu}$.

\begin{proposition} \label{Existence} Let $G$ be a topological group, $\mu\in\mathcal{P}_{\sigma}(G)$ a Baire probability measure, and $(V,\rho)$ a uniformly bounded $\mu$-integrable Banach $G$-module. Then there exists a bounded linear operator $\mathbb{E}_{\mu}$ in the point-weak$^{\ast}$ closure of the convex hull \emph{$\text{conv}\{(\rho_{\mu}^{\ast\ast})^{i}\mid i\in\mathbb{N}\}\subseteq\mathcal{B}(V^{\ast\ast})$} satisfying $\rho_{\mu}^{\ast\ast}\mathbb{E}_{\mu}=\mathbb{E}_{\mu}=\mathbb{E}_{\mu}\rho_{\mu}^{\ast\ast}$, \emph{$\text{im}(\mathbb{E}_{\mu})=\{f\in V^{\ast\ast}\mid\rho_{\mu}^{\ast\ast}(f)=f\}$} and $\mathbb{E}_{\mu}^{2}=\mathbb{E}_{\mu}$. \end{proposition}

\begin{proof} From Remark \ref{BPRemarks} (ii) it follows that every element $T$ in the point-weak$^{\ast}$ closure $\mathcal{C}$ of the convex hull of $\{(\rho_{\mu}^{\ast\ast})^{i}\mid i\in\mathbb{N}\}$ in $\mathcal{B}(V^{\ast\ast})$ satisfies $\Vert T\Vert\leq M$ where $M:=\sup_{g\in G}\Vert\rho_{g}\Vert$. Furthermore, the map 
\[
\mathcal{C}\rightarrow\prod_{f\in V^{\ast\ast}:\Vert f\Vert\leq1}\prod_{\phi\in V^{\ast}:\Vert\phi\Vert\leq1}\mathbb{C}, \, T\mapsto\left(\left(T(f)\right)(\phi)\right)_{f,\phi}
\]
is a homeomorphism onto its image, where the right-hand side is equipped with the product topology. One easily checks that the image of this map is closed and contained in the compact subset $\prod_{f\in V^{\ast\ast}:\Vert f\Vert\leq1}\prod_{\phi\in V^{\ast}:\Vert\phi\Vert\leq1}[-M,M]$. We obtain that $\mathcal{C}\subseteq\mathcal{B}(V^{\ast\ast})$ is point-weak$^{\ast}$ compact. The multiplication map $T\mapsto\rho_{\mu}^{\ast\ast}T$ maps $\mathcal{C}$ to itself with $\rho_{\mu}^{\ast\ast}T=T\rho_{\mu}^{\ast\ast}$ for every $T\in\mathcal{C}$ and is continuous with respect to the subspace topology. The Tychonoff fixed point theorem hence implies the existence of an element $\mathbb{E}_{\mu}\in\mathcal{C}$ with $\rho_{\mu}^{\ast\ast}\mathbb{E}_{\mu}=\mathbb{E}_{\mu}=\mathbb{E}_{\mu}\rho_{\mu}^{\ast\ast}$. It in particular follows that $\text{im}(\mathbb{E}_{\mu})\subseteq\{f\in V^{\ast\ast}\mid\rho_{\mu}^{\ast\ast}(f)=f\}$. Conversely, if $f\in V^{\ast\ast}$ with $\rho_{\mu}^{\ast\ast}(f)=f$, then $T(f)=f$ for every $T\in\text{conv}\{(\rho_{\mu}^{\ast\ast})^{i}\mid i\in\mathbb{N}\}$ via induction and hence also (by a continuity argument) $f=\mathbb{E}_{\mu}(f)\in\text{im}(\mathbb{E}_{\mu})$. By a similar argument we obtain that $\mathbb{E}_{\mu}^{2}=\mathbb{E}_{\mu}$. \end{proof}

For a Baire probability measure $\mu$ on $G$ and a uniformly bounded $\mu$-integrable Banach $G$-module $(V,\rho)$ we define the corresponding \emph{Laplacians} 
\begin{eqnarray*}
\Delta_{\mu} & := & (\text{id}_{V}-\rho_{\mu})\in\mathcal{B}(V),\\
\Delta_{\mu}^{\ast} & := & (\text{id}_{V^{\ast}}-\rho_{\mu}^{\ast})\in\mathcal{B}(V^{\ast}),\\
\Delta_{\mu}^{\ast\ast} & := & (\text{id}_{V^{\ast\ast}}-\rho_{\mu}^{\ast\ast})\in\mathcal{B}(V^{\ast\ast}).
\end{eqnarray*}

\begin{lemma} \label{Inclusion} Let $G$ be a topological group, $\mu\in\mathcal{P}_{\sigma}(G)$ a Baire probability measure, let $(V,\rho)$ be a uniformly bounded $\mu$-integrable Banach $G$-module, and $\mathbb{E}_{\mu}$ an operator as in Proposition \ref{Existence}. We then have that \emph{$\overline{\text{im}(\Delta_{\mu}^{\ast\ast})}^{\Vert\cdot\Vert}\subseteq\ker(\mathbb{E}_{\mu})$} and \emph{$\text{ker}(\mathbb{E}_{\mu})\cap\text{im}(\iota_{V})=\overline{\text{im}(\Delta_{\mu}^{\ast\ast})}^{\Vert\cdot\Vert}\cap\text{im}(\iota_{V})$}. \end{lemma}

\begin{proof} The first inclusion follows from $\rho_{\mu}^{\ast\ast}\mathbb{E}_{\mu}=\mathbb{E}_{\mu}=\mathbb{E}_{\mu}\rho_{\mu}^{\ast\ast}$, as for every $f\in V^{\ast\ast}$ 
\[
\mathbb{E}_{\mu}(\Delta_{\mu}^{\ast\ast}(f))=\mathbb{E}_{\mu}(f)-\mathbb{E}_{\mu}(\rho_{\mu}^{\ast\ast}(f))=\mathbb{E}_{\mu}(f)-\mathbb{E}_{\mu}(f)=0.
\]
By the boundedness of $\mathbb{E}_{\mu}$ it can then be deduced that $\mathbb{E}_{\mu}(f)=0$ for all $f\in\overline{\text{im}(\Delta_{\mu}^{\ast\ast})}^{\Vert\cdot\Vert}$. This also implies that $\overline{\text{im}(\Delta_{\mu}^{\ast\ast})}^{\Vert\cdot\Vert}\cap\text{im}(\iota_{V})\subseteq\ker(\mathbb{E}_{\mu})\cap\text{im}(\iota_{V})$.

For the inclusion $\ker(\mathbb{E}_{\mu})\cap\text{im}(\iota_{V})\subseteq\overline{\text{im}(\Delta_{\mu}^{\ast\ast})}^{\Vert\cdot\Vert}\cap\text{im}(\iota_{V})$ let $\phi$ be an arbitrary non-trivial linear functional on $V^{\ast\ast}$ that vanishes on $\overline{\text{im}(\Delta_{\mu}^{\ast\ast})}^{\Vert\cdot\Vert}$. It follows that $\phi(f)=\phi(\rho_{\mu}^{\ast\ast}(f))$ for every $f\in V^{\ast\ast}$. We have $\phi\circ\iota_{V}\in V^{\ast}$ and then also $\widehat{v}(\phi\circ\iota_{V})=\left(T(\widehat{v})\right)(\phi\circ\iota_{V})$ for every $v\in V$, $T\in\text{conv}\{(\rho_{\mu}^{\ast\ast})^{i}\mid i\in\mathbb{N}\}$. By a continuity argument we obtain that $\widehat{v}(\phi\circ\iota_{V})=\left(T(\widehat{v})\right)(\phi\circ\iota_{V})$ for every $T$ in the point-weak$^{\ast}$ closure of the convex hull. In particular, if $\widehat{v}\in\ker(\mathbb{E}_{\mu})\cap\text{im}(\iota_{V})$, 
\[
\phi(\widehat{v})=\widehat{v}(\phi\circ\iota_{V})=\left(\mathbb{E}_{\mu}(\widehat{v})\right)(\phi\circ\iota_{V})=0.
\]
By the Hahn-Banach theorem (and since $\phi\in V^{\ast\ast\ast}$ with $\phi|_{\overline{\text{im}(\Delta_{\mu}^{\ast\ast})}^{\Vert\cdot\Vert}}\equiv0$ was arbitrary), this means that $\widehat{v}\in\overline{\text{im}(\Delta_{\mu}^{\ast\ast})}^{\Vert\cdot\Vert}$ and therefore $\ker(\mathbb{E}_{\mu})\cap\text{im}(\iota_{V})\subseteq\overline{\text{im}(\Delta_{\mu}^{\ast\ast})}^{\Vert\cdot\Vert}\cap\text{im}(\iota_{V})$. \end{proof}

The lemma implies an ergodic theorem for elements in $\text{im}(\iota_{V})$ whose image under $\mathbb{E}_{\mu}$ is again contained in $\text{im}(\iota_{V})$.

\begin{theorem} \label{ErgodicTheorem} Let $G$ be a topological group, $\mu\in\mathcal{P}_{\sigma}(G)$ a Baire probability measure, let $(V,\rho)$ be a uniformly bounded $\mu$-integrable Banach $G$-module, and $\mathbb{E}_{\mu}$ an operator as in Proposition \ref{Existence}. For every $v\in V$ with \emph{$\mathbb{E}_{\mu}(\widehat{v})\in\text{im}(\iota_{V})$} we then have that 
\[
\mathbb{E}_{\mu}(\widehat{v})=\lim_{n\rightarrow\infty}\frac{1}{n}\sum_{i=0}^{n-1}(\rho_{\mu}^{\ast\ast})^{i}(\widehat{v}),
\]
where the sequence converges in norm. \end{theorem}

\begin{proof} Note that for every $f\in V^{\ast\ast}$ 
\[
\frac{1}{n}\sum_{i=0}^{n-1}(\rho_{\mu}^{\ast\ast})^{i}\left(\Delta_{\mu}^{\ast\ast}(f)\right)=\frac{1}{n}\sum_{i=0}^{n-1}(\rho_{\mu}^{\ast\ast})^{i}(f)-\frac{1}{n}\sum_{i=1}^{n}(\rho_{\mu}^{\ast\ast})^{i}(f)=\frac{1}{n}(\text{id}_{V^{\ast\ast}}-(\rho_{\mu}^{\ast\ast})^{n})(f),
\]
which norm converges to zero, so by continuity $\frac{1}{n}\sum_{k=0}^{n-1}\rho_{\mu^{k}}^{\ast\ast}(f)\rightarrow0$ for every $f\in\overline{\text{im}(\Delta_{\mu}^{\ast\ast})}^{\Vert\cdot\Vert}$. For $v\in V$ with $\mathbb{E}_{\mu}(\widehat{v})\in\text{im}(\iota_{V})$ we have that $\widehat{v}-\mathbb{E}_{\mu}(\widehat{v})\in\ker(\mathbb{E}_{\mu})\cap\text{im}(\iota_{V})$. In combination with Lemma \ref{Inclusion} this gives 
\[
\frac{1}{n}\sum_{i=0}^{n-1}(\rho_{\mu}^{\ast\ast})^{i}(\widehat{v})-\mathbb{E}_{\mu}(\widehat{v})=\frac{1}{n}\sum_{i=0}^{n-1}(\rho_{\mu}^{\ast\ast})^{i}(\widehat{v}-\mathbb{E}_{\mu}(\widehat{v}))\rightarrow0.
\]
This finishes the proof. \end{proof}

For a Banach $G$-module $(V,\rho)$ we write $V^{G}$ for the corresponding \emph{subspace of fixed points} under $\rho$. Furthermore, denote the distance of an element $v\in V$ to a closed subspace $W\subseteq V$ by $\text{dist}(v,W)$, i.e. $\text{dist}(v,W):=\inf\{\Vert v-w\Vert\mid w\in W\}$. \\

The proof of the following theorem uses the argument in \cite[Theorem 5.1]{HartmanKalantar} (which again relies on the proof of Furstenberg's conjecture by Kaimanovich and Vershik in \cite[Theorem 4.3]{KaimanovichVershik}). We will make use of it in the later sections. We do not know whether the statement generalizes to non-separable Banach spaces.

\begin{theorem} \label{FixedPointApproximation} Let $G$ be a topological group and $(V,\rho)$ a uniformly bounded weakly operator continuous Banach $G$-module with $V$ being separable (so that $(V,\rho)$ is in particular BP-integrable). Furthermore, assume that for any $\varepsilon>0$, $v\in V$ there exists a Baire probability measure $\mu\in\mathcal{P}_{\sigma}(G)$ such that \emph{$\text{dist}\left(\rho_{\mu}(v),V^{G}\right)<\varepsilon$}. Then there also exists $\mu\in\mathcal{P}_{\sigma}(G)$ with $\mathbb{E}_{\mu}(\widehat{v})\in(V^{\ast\ast})^{G}$ for every $v\in V$, where $\mathbb{E}_{\mu}$ is an operator as in Proposition \ref{Existence}. \end{theorem}

\begin{proof} Let $(v_{i})_{i\in\mathbb{N}}\subseteq V$ be a sequence for which $\{v_{1},v_{2},...\}$ is dense in the unit ball of $V$. Furthermore, fix an arbitrary sequence $(n_{i})_{i\in\mathbb{N}}\subseteq\mathbb{N}_{\geq1}$ with $(\sum_{j=1}^{i}2^{-j})^{n_{i}}<2^{-i}$ for all $i\in\mathbb{N}_{\geq1}$. By our assumption we can inductively choose Baire probability measures $(\mu_{i})_{i\in\mathbb{N}}\in\mathcal{P}_{\sigma}(G)$ such that for all $i\in\mathbb{N}_{\geq1}$, $1\leq j,j_{1},...,j_{k}<i$ and $k<n_{i}$ the inequality 
\[
\text{dist}\left(\rho_{\mu_{i}}\rho_{\mu_{j_{k}}}...\rho_{\mu_{j_{1}}}(v_{j}),V^{G}\right)<2^{-i}
\]
holds. We claim that $\mu:=\sum_{i=1}^{\infty}2^{-i}\mu_{i}\in\mathcal{P}_{\sigma}(G)$ satisfies $\text{dist}\left(\rho_{\mu}^{i}(v),V^{G}\right)\rightarrow0$ for every $v\in V$. Indeed, for $v\in V$ with $\Vert v\Vert\leq1$ we find $i\in\mathbb{N}$ with $\Vert v-v_{i}\Vert<2^{-i}$. Then, by Remark \ref{BPRemarks} (ii), 
\begin{eqnarray*}
\text{dist}\left(\rho_{\mu}^{n_{i}}(v),V^{G}\right) & \leq & \Vert\rho_{\mu}^{n_{i}}(v)-\rho_{\mu}^{n_{i}}(v_{i})\Vert+\text{dist}\left(\rho_{\mu}^{n_{i}}(v_{i}),V^{G}\right)\\
 & \leq & 2^{-i}M+\text{dist}\left(\rho_{\mu}^{n_{i}}(v_{i}),V^{G}\right),
\end{eqnarray*}
where $M:=\sup_{g\in G}\Vert\rho_{g}\Vert$. The expression $\text{dist}\left(\rho_{\mu}^{n_{i}}(v_{i}),V^{G}\right)$ can be estimated via 
\begin{eqnarray*}
\text{dist}\left(\rho_{\mu}^{n_{i}}(v_{i}),V^{G}\right) & = & \text{dist}\left(\sum_{j_{1},...,j_{n_{i}}=1}^{\infty}\frac{1}{2^{j_{1}+...+j_{n_{i}}}}\rho_{\mu_{j_{1}}}...\rho_{\mu_{j_{n_{i}}}}(v_{i}),V^{G}\right)\\
 & \leq & \sum_{\max(j_{1},...,j_{n_{i}})\leq i}\frac{1}{2^{j_{1}+...+j_{n_{i}}}}\left\Vert \rho_{\mu_{j_{1}}}...\rho_{\mu_{j_{n_{i}}}}(v_{i})\right\Vert \\
 &  & \qquad+\sum_{\max(j_{1},...,j_{n_{i}})>i}\frac{1}{2^{j_{1}+...+j_{n_{i}}}}\text{dist}\left(\rho_{\mu_{j_{1}}}...\rho_{\mu_{j_{n_{i}}}}(v_{i}),V^{G}\right)\\
 & \leq M & \left(\sum_{j=1}^{i}2^{-j}\right)^{n_{i}}\left\Vert v_{i}\right\Vert +\sum_{\max(j_{1},...,j_{n_{i}})>i}\frac{1}{2^{j_{1}+...+j_{n_{i}}}}\text{dist}\left(\rho_{\mu_{j_{1}}}...\rho_{\mu_{j_{n_{i}}}}(v_{i}),V^{G}\right).
\end{eqnarray*}
For every tuple $(j_{1},...,j_{n_{i}})$ with $\max(j_{1},...,j_{n_{i}})>i$ there exists a largest integer $1\leq k\leq n_{i}$ for which $j_{k}>i$. Hence, by the choice of the sequence $(\mu_{j})_{j\in\mathbb{N}}\in\mathcal{P}_{\sigma}(G)$, 
\[
\text{dist}\left(\rho_{\mu_{j_{1}}}...\rho_{\mu_{j_{n_{i}}}}(v_{i}),V^{G}\right)\leq M\text{dist}\left(\rho_{\mu_{j_{k}}}...\rho_{\mu_{j_{n_{i}}}}(v_{i}),V^{G}\right)<2^{-j_{k}}M<2^{-i}M.
\]
Here the first inequality follows from the fact that the representation $\rho$ induces a (well-defined) linear representation $\widetilde{\rho}:G\rightarrow\text{GL}(V/V^{G})$ on the quotient space with uniform bound $M$. With $(\sum_{j=1}^{i}2^{-j})^{n_{i}}<2^{-i}$ and $\Vert v_{i}\Vert\leq1$ we obtain $\text{dist}\left(\rho_{\mu}^{n_{i}}(v_{i}),V^{G}\right)<2^{-i}(1+M)$, thus 
\[
\text{dist}\left(\rho_{\mu}^{n_{i}}(v),V^{G}\right)<2^{-i}(1+2M).
\]
For $n>n_{i}$ this implies 
\[
\text{dist}\left(\rho_{\mu}^{n}(v),V^{G}\right)\leq M\text{dist}\left(\rho_{\mu}^{n_{i}}(v),V^{G}\right)\leq2^{-i}M(1+2M)
\]
and therefore $\text{dist}\left(\rho_{\mu}^{n}(v),V^{G}\right)\rightarrow0$. We can hence choose a sequence $(w_{n})_{n\in\mathbb{N}}\subseteq V^{G}$ with $\Vert\rho_{\mu}^{n}(v)-w_{n}\Vert\rightarrow0$. By Remark \ref{BPRemarks} (ii) we have $\Vert\rho_{\mu}^{n}(v)\Vert\leq M\Vert v\Vert$ for every $n\in\mathbb{N}$. This gives that the sequence $(\widehat{w_{n}})_{n\in\mathbb{N}}\subseteq V^{\ast\ast}$ is bounded so that we can go over to a subnet $(\widehat{w_{n_{\lambda}}})_{\lambda\in\Lambda}$ with weak$^{\ast}$-limit $f\in V^{\ast\ast}$. It is easy to check that $f\in(V^{\ast\ast})^{G}$ and hence 
\[
\mathbb{E}_{\mu}(\widehat{v})-f=\lim_{\lambda}(\mathbb{E}_{\mu}\circ\iota_{V})\left(\rho_{\mu}^{n_{\lambda}}(v)-w_{n_{\lambda}}\right)=0.
\]
We obtain that $\mathbb{E}_{\mu}(\widehat{v})=f\in(V^{\ast\ast})^{G}$, as claimed. \end{proof}

\begin{remark} \emph{(i)} If $V^{G}$ is finite-dimensional, one can employ the norm compactness of the unit ball to deduce that, under the same assumptions as in Theorem \ref{FixedPointApproximation}, $\mathbb{E}_{\mu}(\widehat{v})\in\iota_{V}(V^{G})$ for every $v\in V$.

\emph{(ii)} Let $G$ be a topological group and $(V,\rho)$ a uniformly bounded BP-integrable Banach $G$-module. Assume that $V^{\ast}$ is separable and that for any $\varepsilon>0$, $\phi\in V^{\ast}$ there exists a Baire probability measure $\mu\in\mathcal{P}_{\sigma}(G)$ such that $\text{dist}\left(\rho_{\mu}^{\ast}(\phi),(V^{\ast})^{G}\right)<\varepsilon$. A simple modification of the proof of Theorem \ref{FixedPointApproximation} then implies that there exists $\mu\in\mathcal{P}_{\sigma}(G)$ with $\mathbb{E}_{\mu}(\widehat{v})\in(V^{\ast\ast})^{G}$ for every $v\in V$, where $\mathbb{E}_{\mu}$ is an operator as in Proposition \ref{Existence}. \end{remark}

\begin{example} A unital C$^{\ast}$-algebra $A$ with center $\mathcal{Z}(A):=\{a\in A\mid ab=ba\text{ for all }b\in A\}$ is said to have the \emph{Dixmier property} if for every $a\in A$ the norm closure of the convex hull of $\{uau^{\ast}\mid u\in\mathcal{U}(A)\}$ intersects $\mathcal{Z}(A)$ non-trivially. Here $\mathcal{U}(A)$ denotes the unitary group of $A$. If we equip $\mathcal{U}(A)$ with the restriction of the weak topology on $A$, we obtain a weakly operator continuous uniformly bounded linear representation $\rho:\mathcal{U}(A)\rightarrow\text{Aut}(A)\subseteq\text{GL}(A)$ via $\rho_{u}(a):=uau^{\ast}$. Note that $\mathcal{Z}(A)=A^{\mathcal{U}(A)}$ so that, in the case where $A$ is separable and has the Dixmier property, Theorem \ref{FixedPointApproximation} implies the existence of a Baire probability measure $\mu\in\mathcal{P}_{\sigma}(\mathcal{U}(A))$ with $\mathbb{E}_{\mu}(\widehat{a})\in(A^{\ast\ast})^{\mathcal{U}(A)}$ for every $a\in A$, with $\mathbb{E}_{\mu}$ as in Proposition \ref{Existence}. This class of examples in particular covers all separable simple unital C$^{\ast}$-algebras with at most one tracial state, see \cite{HaagerupSzido}. \end{example}

In a similar spirit, we obtain the following result.

Given a Baire probability measure $\mu$ on a topological group $G$ denote the closed subgroup of $G$ generated by the set $\Sigma_{\mu}(G)$ of elements $g\in G$ for which $\int_{G}fd\mu>0$ for every positive function $f\in C_{b}(G,\mathbb{R})$ with $f(g)\neq0$ by $G_{\mu}$. Furthermore, call a function $f:G\rightarrow\mathbb{C}$ on $G$ \emph{right-unifomly continuous} if for every $\varepsilon>0$ one can find a neighbourhood $\mathcal{U}$ of $e\in G$ with $|f(g)-f(h)|\leq\varepsilon$ for all $g,h\in G$ with $hg^{-1}\in\mathcal{U}$. The set $\text{RUCB}(G)$ of all right-uniformly continuous functions that are bounded with respect to the supremum norm becomes a commutative unital C$^{\ast}$-algebra with respect to pointwise addition, multiplication, and involution. For every $f\in\text{RUCB}(G)$ and $\mu\in\mathcal{P}_{\sigma}(G)$ the function $\Phi_{\mu}f:g\mapsto\int_{G}f(gh)d\mu(h)$ is again contained in $\text{RUCB}(G)$. The function $f$ is called \emph{$\mu$-harmonic}, if $f=\Phi_{\mu}f$. We say that the pair $(G,\mu)$ has the \emph{Liouville property} if for any $g\in G$ every $\mu$-harmonic function $f\in\text{RUCB}(G)$ satisfies $f(gh)=f(g)$ for $\mu$-almost all $h\in G$. The group $G$ is furthermore said to be \emph{Choquet--Deny} if $(G,\mu)$ satisfies the Liouville property for every $\mu\in\mathcal{P}_{\sigma}(G)$. Note that there are several different definitions of these concepts in the literature. Here we employ the ones that suit our purposes.

\begin{proposition} \label{LiouvilleFixedPoint} Let $G$ be a topological group and $\mu\in\mathcal{P}_{\sigma}(G)$ a Baire probability measure for which the pair $(G,\check{\mu})$ has the Liouville property. Let furthermore $(V,\rho)$ be a uniformly bounded, strongly operator continuous, $\mu$-integrable Banach $G$-module. Then $\mathbb{E}_{\mu}\circ\iota_{V}$ with $\mathbb{E}_{\mu}$ as in Proposition \ref{Existence} is $G_{\mu}$-invariant. \end{proposition}

\begin{proof} For fixed $\phi\in V^{\ast}$ we can define a bounded linear functional $\psi_{\phi}\in V^{\ast}$ via $\psi_{\phi}(v):=\mathbb{E}_{\mu}(\widehat{v})(\phi)$. For every $v\in V$ the map $f_{v,\phi}:G\rightarrow\mathbb{C}$, $f_{v,\phi}(g):=\rho_{g}^{\ast}(\psi_{\phi})(v)$ is contained in $\text{RUCB}(G)$ with 
\[
\int_{G}f_{v,\phi}(gh)d\breve{\mu}(h)=\int_{G}\rho_{gh}^{\ast}(\psi_{\phi})(v)d\breve{\mu}(h)=(\psi_{\phi}\circ\rho_{\mu})\left(\rho_{g^{-1}}(v)\right)=\rho_{g}^{\ast}(\psi_{\phi})(v)=f_{v,\phi}(g).
\]
It follows that $f_{v,\phi}$ is $\breve{\mu}$-harmonic and hence for all $g\in G$, $f_{v,\phi}(gh)=f_{v,\phi}(g)$ for $\mu$-almost every $h\in G$. We in particular obtain that the function $h\mapsto|f_{v,\phi}(h)-f_{v,\phi}(e)|$ is positive, bounded and continuous with $\int_{G}|f_{v,\phi}(h)-f_{v,\phi}(e)|d\mu(h)=0$. From the definition of $\Sigma_{\mu}(G)$ it follows that $f_{v,\phi}(h)=f_{v,\phi}(e)$ for all $h\in\Sigma_{\mu}(G)$ and hence, since $\phi\in V^{\ast}$, $v\in V$ where chosen arbitrarily, $\mathbb{E}_{\mu}\circ\iota_{V}\circ\rho_{h}=\mathbb{E}_{\mu}\circ\iota_{V}$. We finally deduce that $\mathbb{E}_{\mu}\circ\iota_{V}$ is $G_{\mu}$-invariant, as claimed. \end{proof}

\begin{remark} \label{LiouvilleRemark} The characterization in \cite[Corollary 4.8]{SchneiderThom} states that a Hausdorff second-countable topological group $G$ is \emph{amenable} (i.e. the canonical action of $G$ on the state space of $\text{RUCB}(G)$ has a fixed point) if and only if it admits a fully supported, regular Borel probability measure $\mu$ for which every $\mu$-harmonic function is constant. In combination with Proposition \ref{LiouvilleFixedPoint} this implies that every amenable topological group $G$ admits a Baire probability measure $\mu\in\mathcal{P}_{\sigma}(G)$ such that for every uniformly bounded, strongly operator continuous, BP-integrable Banach $G$-module $(V,\rho)$ the map $\mathbb{E}_{\mu}\circ\iota_{V}$ is $G$-invariant. \end{remark}

\begin{corollary} \label{CenterFixedPoint} Let $G$ be a topological group, $\mu\in\mathcal{P}_{\sigma}(\mathcal{Z}(G))$ a Baire probability measure, $(V,\rho)$ a uniformly bounded, strongly operator continuous, $\mu$-integrable Banach $G$-module, and set $Z:=Z(G)_{\mu}$. Assume that the product map $Z\times Z\rightarrow Z$ is measurable with respect to the product $\sigma$-algebra $\mathcal{B}a(Z)\otimes\mathcal{B}a(Z)$ and $\mathcal{B}a(Z)$. Then $\mathbb{E}_{\mu}\circ\iota_{V}$ with $\mathbb{E}_{\mu}$ as in Proposition \ref{Existence} is $Z$-invariant with \emph{$\text{im}(\mathbb{E}_{\mu}\circ\iota_{V})\subseteq(V^{\ast\ast})^{Z}$}. \end{corollary}

\begin{proof} By Proposition \ref{LiouvilleFixedPoint} it suffices to show that the center of $G$ is Choquet--Deny. Following the proof in \cite{Prunaru00}, consider for $\mu\in\mathcal{P}_{\sigma}(\mathcal{Z}(G))$ the bounded linear operator $\Phi_{\mu}$ on $\text{RUCB}(\mathcal{Z}(G))$ defined above and define for fixed $f\in\text{RUCB}(\mathcal{Z}(G))$ with $f=\Phi_{\mu}f$ a function $h\in\text{RUCB}(\mathcal{Z}(G))$ via $h(z):=\int_{\mathcal{Z}(G)}|f(z)-f(cz)|^{2}d\mu(c)$. In combination with Fubini's theorem and the Cauchy-Schwarz inequality we obtain that $\Phi_{\mu}(h)(z)\geq h(z)$ for all $z\in\mathcal{Z}(G)$. By making use of the identity $h=\Phi_{\mu}(|f|^{2})-|f|^{2}$, we then have 
\[
h\leq\frac{1}{N}\sum_{i=0}^{N}\left(\Phi_{\mu}^{i+1}(|f|^{2})-\Phi_{\mu}^{i}(|f|^{2})\right)=\frac{1}{N}\left(\Phi_{\mu}^{N+1}(|f|^{2})-1\right)\rightarrow0,
\]
so that $h$ vanishes. It follows that for every $z\in\mathcal{Z}(G)$ the equality $f(z)=f(cz)$ holds $\mu$-almost everywhere. This finishes the proof. \end{proof}

\vspace{3mm}



\section{The projection $\mathbb{E}_{\mu}$\label{sec:The-projection}}

Theorem \ref{ErgodicTheorem} indicates that it is desirable to know more about the range of the map $\mathbb{E}_{\mu}$ in Proposition \ref{Existence}. Studying it will be the aim of this section.

\vspace{3mm}


\subsection{Weakly uniquely stationary representations}

For $\mu\in\mathcal{P}_{\sigma}(G)$ and a uniformly bounded $\mu$-integrable linear representation $\rho:G\rightarrow\text{GL}(V)$ of a topological group $G$ on a Banach space $V$ we write 
\[
V^{\mu}:=\{v\in V\mid\rho_{\mu}(v)=v\}\subseteq V
\]
for the subspace of all \emph{$\rho_{\mu}$-invariant elements} in $V$. Similarly, $(V^{\ast})^{\mu}\subseteq V^{\ast}$ denotes the space of all \emph{$\rho_{\mu}^{\ast}$-invariant functionals} on $V$ and $(V^{\ast\ast})^{\mu}$ is the space of all \emph{$\rho_{\mu}^{\ast\ast}$-invariant elements} in $V^{\ast\ast}$. The equality (\ref{eq:Compatibility}) immediately implies that $\iota_{V}(V^{\mu})\subseteq(V^{\ast\ast})^{\mu}$, whereas Proposition \ref{Existence} in particular gives that $\text{im}(\mathbb{E}_{\mu})=(V^{\ast\ast})^{\mu}$.

\begin{definition} \label{WeaklyUniquelyStationary} Let $G$ be a topological group and $\mu\in\mathcal{P}_{\sigma}(G)$ a Baire probability measure. We call a uniformly bounded linear representation $\rho:G\rightarrow\text{GL}(V)$ of $G$ on a Banach space $V$ \emph{weakly uniquely $\mu$-stationary} if it is $\mu$-integrable and if every non-trivial functional $\phi\in(V^{\ast})^{\breve{\mu}}$ admits an element $v\in V^{\mu}$ with $\phi(v)\neq0$. In this case we simply say that the Banach $G$-module $(V,\rho)$ is \emph{weakly uniquely $\mu$-stationary}. \end{definition}

The first three statements of the following proposition are essentially contained in \cite{Kakutani38}, \cite{Yosida38}, \cite{Sine70} (see also \cite[Chapter 2.1]{Krengel85}). We include a proof for the convenience of the reader.

\begin{proposition} \label{EquivalentCharacterizations} Let $G$ be a topological group. For a Baire probability measure $\mu\in\mathcal{P}_{\sigma}(G)$ and a uniformly bounded $\mu$-integrable Banach $G$-module $(V,\rho)$ the following statements are equivalent: 
\begin{enumerate}
\item $\rho$ is weakly uniquely $\mu$-stationary; 
\item $V$ admits a decomposition into norm closed $\rho_{\mu}$-invariant linear subspaces \emph{$V=\overline{\text{im}(\Delta_{\mu})}\oplus V^{\mu}$}; 
\item The sequence $\left(\frac{1}{n}\sum_{i=0}^{n-1}(\rho_{\mu}^{\ast\ast})^{i}(\widehat{v})\right)_{n\in\mathbb{N}}\subseteq V^{\ast\ast}$ converges for every $v\in V$ in norm; 
\item \emph{$\mathbb{E}_{\mu}(\widehat{v})\in\text{im}(\iota_{V})$} for every $v\in V$, where $\mathbb{E}_{\mu}$ is an operator as in Proposition \ref{Existence}. 
\end{enumerate}
\end{proposition}

\begin{proof} ``$\text{(1)}\Rightarrow\text{(2)}$'' Assume that the representation $\rho$ is weakly uniquely $\mu$-stationary. Both $\overline{\text{im}(\Delta_{\mu})}\subseteq V$ and $V^{\mu}\subseteq V$ are obviously norm closed. As in the proof of Theorem \ref{ErgodicTheorem} it can be deduced that $\frac{1}{n}\sum_{i=0}^{n-1}(\rho_{\mu}^{\ast\ast})^{i}(\widehat{v})\rightarrow0$ in norm for every $v\in\overline{\text{im}(\Delta_{\mu})}$, whereas $\frac{1}{n}\sum_{i=0}^{n-1}(\rho_{\mu}^{\ast\ast})^{i}(\widehat{v})=\widehat{v}$ for all $n\in\mathbb{N}$, $v\in V^{\mu}$. Hence, $\overline{\text{im}(\Delta_{\mu})}\cap V^{\mu}=\{0\}$. Define $V_{0}:=\overline{\text{im}(\Delta_{\mu})}\oplus V^{\mu}\subseteq V$ and let $\phi\in V^{\ast}$ be a functional on $V$ that vanishes on $V_{0}$. One in particular has that $\phi\circ\rho_{\mu}(v)=\phi(v)$ for all $v\in V$, i.e. $\phi\in(V^{\ast})^{\breve{\mu}}$. But $\rho$ is weakly uniquely $\mu$-stationary, so by construction, the functional $\phi$ must be trivial. The Hahn-Banach theorem finally implies that $V_{0}=V$.

``$\text{(2)}\Rightarrow\text{(3)}$'' Assume that $V=\overline{\text{im}(\Delta_{\mu})}\oplus V^{\mu}$ and let $v=v_{1}+v_{2}$ where $v_{1}\in\overline{\text{im}(\Delta_{\mu})}$ and $v_{2}\in V^{\mu}$. As we have seen before, $\frac{1}{n}\sum_{i=0}^{n-1}(\rho_{\mu}^{\ast\ast})^{i}(\widehat{v_{1}})\rightarrow0$ and $\frac{1}{n}\sum_{i=0}^{n-1}(\rho_{\mu}^{\ast\ast})^{i}(\widehat{v_{2}})\rightarrow\widehat{v_{2}}$ in norm. It follows that the sequence $\left(\frac{1}{n}\sum_{i=0}^{n-1}(\rho_{\mu}^{\ast\ast})^{i}(\widehat{v})\right)_{n\in\mathbb{N}}\subseteq V^{\ast\ast}$ norm-converges for every $v\in V$.

``$\text{(3)}\Rightarrow\text{(4)}$'' Assume that the sequence $\left(\frac{1}{n}\sum_{i=0}^{n-1}(\rho_{\mu}^{\ast\ast})^{i}(\widehat{v})\right)_{n\in\mathbb{N}}\subseteq V^{\ast\ast}$ converges for every $v\in V$ in norm. The limit of this sequence is $\rho_{\mu}^{\ast\ast}$-invariant so that 
\begin{eqnarray*}
\mathbb{E}_{\mu}(\widehat{v}) & = & \lim_{n\rightarrow\infty}\frac{1}{n}\sum_{i=0}^{n-1}\mathbb{E}_{\mu}\left((\rho_{\mu}^{\ast\ast})^{i}(\widehat{v})\right)\\
 & = & \mathbb{E}_{\mu}\left(\lim_{n\rightarrow\infty}\frac{1}{n}\sum_{i=0}^{n-1}(\rho_{\mu}^{\ast\ast})^{i}(\widehat{v})\right)\\
 & = & \lim_{n\rightarrow\infty}\frac{1}{n}\sum_{i=0}^{n-1}(\rho_{\mu}^{\ast\ast})^{i}(\widehat{v})\,
\end{eqnarray*}
is contained in $\text{im}(\iota_{V})$ as claimed.

``$\text{(4)}\Rightarrow\text{(1)}$'' Assume that $\mathbb{E}_{\mu}(\widehat{v})\in\text{im}(\iota_{V})$ for every $v\in V$. By Theorem \ref{ErgodicTheorem} we then have that for $v\in V$ and $\phi\in(V^{\ast})^{\breve{\mu}}$ 
\begin{eqnarray*}
\phi(v) & = & \lim_{n\rightarrow\infty}\frac{1}{n}\sum_{i=0}^{n-1}\left(\phi\circ\rho_{\mu}^{i}\right)(v)\\
 & = & \lim_{n\rightarrow\infty}\left(\phi\circ\iota_{V}^{-1}\right)\left(\frac{1}{n}\sum_{i=0}^{n-1}(\rho_{\mu}^{\ast\ast})^{i}(\widehat{v})\right)\\
 & = & \left(\phi\circ\iota_{V}^{-1}\right)\left(\mathbb{E}_{\mu}(\widehat{v})\right).
\end{eqnarray*}
But $(\iota_{V}^{-1}\circ\mathbb{E}_{\mu})(\widehat{v})\in V^{\mu}$. Hence $\phi$ must be trivial if $\phi(v)=0$ for all $v\in V^{\mu}$. It follows that $\rho$ is weakly uniquely $\mu$-stationary. \end{proof}

Justified by Proposition \ref{EquivalentCharacterizations}, for a Baire probability measure $\mu\in\mathcal{P}_{\sigma}(G)$ and a uniformly bounded weakly uniquely $\mu$-stationary Banach $G$-module $(V,\rho)$, we will from now on denote $\widetilde{\mathbb{E}}_{\mu}:=\iota_{V}^{-1}\circ\mathbb{E}_{\mu}\circ\iota_{V}$.

The following class of examples occurs in \cite[Proposition 2.2]{JaworskiNeufang}. Representations like this have been studied in \cite{CTV08}. More examples can be found in Lemma \ref{Connection}, in Subsection \ref{subsec:Weakly-almost-periodic}, and in \cite[Section 2]{JaworskiNeufang}.

\begin{example} Let $G$ be a countable discrete group equipped with a probability measure $\mu$ satisfying $\#G_{\mu}=\infty$, and let $(V,\rho)$ be a uniformly bounded Banach $G$-module. Suppose that for all $v\in V$, $\phi\in V^{\ast}$ the function $f_{\phi,v}:g\mapsto\rho_{g}^{\ast}(\phi)(v)$ is contained in $C_{0}(G)$. Then $(V,\rho)$ is weakly uniquely $\mu$-stationary. Indeed, by a similar argument as in Proposition \ref{LiouvilleFixedPoint} one deduces that for $\phi\in(V^{\ast})^{\breve{\mu}}$ and $v\in V$ the function $f_{\phi,v}\in C_{0}(G)$ is $\breve{\mu}$-harmonic. The maximum principle (see e.g. \cite[Exercise 2.80]{Yadin24}) then implies that $\phi\equiv0$ and hence $(V^{\ast})^{\breve{\mu}}=0$. \end{example}

\begin{proposition} \label{ConvexHull} Let $G$ be a topological group, $\mu\in\mathcal{P}_{\sigma}(G)$ a Baire probability measure and $(V,\rho)$ a uniformly bounded, weakly uniquely $\mu$-stationary Banach $G$-module. Then all finite subsets $T\subseteq V$ and $\varepsilon>0$ admit an element \emph{$\delta\in\text{conv}\{\rho_{g}\mid g\in G\}\subseteq\mathcal{B}(V)$} with $\sup_{v\in T}\Vert\widetilde{\mathbb{E}}_{\mu}(v)-\delta(v)\Vert<\varepsilon$. \end{proposition}

\begin{proof} Let $\{v_{1},...,v_{n}\}\subseteq V$ be a finite subset and $\varepsilon>0$. By Proposition \ref{EquivalentCharacterizations} there exists an integer $N\in\mathbb{N}$ with $\max_{1\leq i\leq n}\Vert\widetilde{\mathbb{E}}_{\mu}(v_{i})-\frac{1}{N}\sum_{l=0}^{N-1}\rho_{\mu}^{l}(v_{i})\Vert<\frac{\varepsilon}{2}$. Set $\mathbf{v}:=(v_{1},...,v_{n})\in V^{n}$ and consider the Banach space $V^{n}:=\bigoplus_{1\leq i\leq n}V$ with the norm $\Vert(w_{1},...,w_{n})\Vert_{n}:=\max_{1\leq i\leq n}\Vert w_{i}\Vert$ and the diagonal representation $\overline{\rho}:G\rightarrow\text{GL}(V^{n})$. Then $(V^{n},\overline{\rho})$ is a weakly operator continuous uniformly bounded Banach $G$-module. By $(V^{n})^{\ast}\cong(V^{\ast})^{n}$ we furthermore have that $(V^{n},\overline{\rho})$ is weakly uniquely $\mu$-stationary with $\overline{\rho}_{\mu}(\mathbf{w})=(\rho_{\mu}(w_{1}),...,\rho_{\mu}(w_{n}))$ for all $\mathbf{w}:=(w_{1},...,w_{n})\in V^{n}$. Remark \ref{BPRemarks} (ii) implies the existence of elements $\overline{\delta}_{0},...,\overline{\delta}_{N-1}\in\text{conv}\{\overline{\rho}_{g}\mid g\in G\}$ with $\max_{1\leq l\leq N-1}\Vert\overline{\rho}_{\mu}^{l}(\mathbf{v})-\overline{\delta}_{l}(\mathbf{v})\Vert_{n}<\frac{\varepsilon}{2}$. Denote the corresponding elements in $\text{conv}\{\rho_{g}\mid g\in G\}$ by $\delta_{0},...,\delta_{N-1}$ and set $\delta:=\frac{1}{N}\sum_{l=0}^{N-1}\delta_{l}$. Then, 
\begin{eqnarray*}
 &  & \max_{1\leq i\leq n}\Vert\widetilde{\mathbb{E}}_{\mu}(v_{i})-\delta(v_{i})\Vert\\
 & \leq & \frac{\varepsilon}{2}+\frac{1}{N}\sum_{l=0}^{N-1}\left(\max_{1\leq i\leq n}\Vert\rho_{\mu}^{l}(v_{i})-\delta_{l}(v_{i})\Vert\right)\\
 & = & \frac{\varepsilon}{2}+\frac{1}{N}\sum_{l=0}^{N-1}\Vert\overline{\rho}_{\mu}^{l}(\mathbf{v})-\overline{\delta}_{l}(\mathbf{v})\Vert_{n}\\
 & < & \varepsilon.
\end{eqnarray*}
This finishes the proof. \end{proof}

Definition \ref{WeaklyUniquelyStationary} is inspired by the notion of uniquely stationary actions of groups on C$^{\ast}$-algebras, as introduced by Hartman and Kalantar in \cite{HartmanKalantar}. Recall that for an action $\alpha:G\rightarrow\text{Aut}(A)$ of a countable discrete group $G$ on a unital C$^{\ast}$-algebra $A$ by automorphism and $\mu\in\text{Prob}(G)$ a state $\tau$ on $A$ is said to be \emph{$\mu$-stationary}, if $\tau=\sum_{g\in G}\mu(g)\alpha_{g}^{\ast}(\tau)$. By \cite[Proposition 4.2]{HartmanKalantar} every action $(G,\mu)\curvearrowright A$ admits a $\mu$-stationary state. It is called \emph{uniquely stationary} if there exists a unique $\mu$-stationary state on $A$.

Examples of uniquely stationary actions can be found in \cite{HartmanKalantar} and \cite{BaderBoutonnetHoudayerPeterson}. For instance, by \cite[Theorem 5.1]{HartmanKalantar} the reduced group C$^{\ast}$-algebra $C_{r}^{\ast}(G)$ of a countable discrete group $G$ is \emph{simple} (i.e. it contains no non-trivial two-sided closed ideal) if and only if there exists a probability measure $\mu$ on $G$ such that the action $(G,\mu)\curvearrowright C_{r}^{\ast}(G)$ by inner automorphisms is uniquely stationary. Similarly, \cite[Proposition 3.7]{BaderBoutonnetHoudayerPeterson} states that for any non-amenable charmenable countable discrete group $G$ with trivial amenable radical and any unitary representation $\pi:G\rightarrow\mathcal{U}(\mathcal{H})$ of $G$ there exists a probability measure $\mu$ on $G$ whose support generates the whole group and such that the C$^{\ast}$-algebra generated by $\pi$ carries a unique $\mu$-stationary state. \\

The purpose of the following lemma is to clear up the connection between our notion and the one of Hartman and Kalantar.

\begin{lemma} \label{Connection} Let $A$ be a unital C$^{\ast}$-algebra and $\alpha:G\rightarrow\text{Aut}(A)$ an action of a countable discrete group $G$ by automorphism. Let furthermore $\mu$ be a probability measure on $G$. Then the following statements hold: 
\begin{enumerate}
\item If $(G,\breve{\mu})\curvearrowright A$ is uniquely stationary, then the isometric linear representation $g\mapsto\alpha_{g}$ is weakly uniquely $\mu$-stationary; 
\item If $A^{\mu}=\mathbb{C}1$, then $(G,\breve{\mu})\curvearrowright A$ is uniquely stationary if and only if $g\mapsto\alpha_{g}$ is weakly uniquely $\mu$-stationary. 
\end{enumerate}
\end{lemma}

\begin{proof} \emph{About (1)}: Assume that $(G,\breve{\mu})\curvearrowright A$ is uniquely stationary and let $\tau\in\mathcal{S}(A)$ be the corresponding $\breve{\mu}$-stationary state. Since $G$ is countable, the representation $g\mapsto\alpha_{g}$ is BP-integrable. The uniqueness of the Jordan decomposition furthermore implies that every functional $\phi\in(A^{\ast})^{\breve{\mu}}\setminus\{0\}$ decomposes as $\phi=\tau_{+}-\tau_{-}$, where both $\tau_{+}$ and $\tau_{-}$ are $\breve{\mu}$-stationary positive functionals on $A$. But then, by our assumption, $\phi$ must be a non-trivial multiple of $\tau$, so in particular $\phi(1)\neq0$. It follows that the representation is weakly uniquely $\mu$-stationary.

\emph{About (2):} Assume that $A^{\mu}=\mathbb{C}1$. It remains to show that, if $\alpha : g\mapsto\alpha_{g}$ is weakly uniquely $\mu$-stationary, the action $(G,\breve{\mu})\curvearrowright A$ is uniquely stationary. So assume that $g\mapsto\alpha_{g}$ is weakly uniquely $\mu$-stationary. By Proposition \ref{EquivalentCharacterizations}, $\widetilde{\mathbb{E}}_{\mu}$ defines a well-defined $\breve{\mu}$-stationary state on $A$. Assume that $\tau\in\mathcal{S}(A)$ is another such state, then (by a similar calculation as in the proof of Proposition \ref{EquivalentCharacterizations}), 
\begin{eqnarray*}
\tau(a) & = & \lim_{n\rightarrow\infty}\frac{1}{n}\sum_{i=0}^{n-1}\left(\tau\circ\alpha_{\mu}^{i}\right)(a)=\lim_{n\rightarrow\infty}\left(\tau\circ\iota_{A}^{-1}\right)\left(\frac{1}{n}\sum_{i=0}^{n-1}(\alpha_{\mu}^{\ast\ast})^{i}(\widehat{a})\right)\\
 & = & \left(\tau\circ\iota_{A}^{-1}\right)\left(\mathbb{E}_{\mu}(\widehat{a})\right)=\widetilde{\mathbb{E}}_{\mu}(a)
\end{eqnarray*}
for all $a\in A$. \end{proof}

\vspace{3mm}


\subsection{Weakly almost periodic representations\label{subsec:Weakly-almost-periodic}}

In \cite{BaderRosendalSauer} the study of reduced continuous cohomology of weakly almost periodic representations was initiated.

\begin{definition}[{\cite[Definition 1]{BaderRosendalSauer}}] A linear representation $\rho:G\rightarrow\text{GL}(V)$ of a topological group $G$ on a Banach space $V$ is called \emph{weakly almost periodic} (or simply \emph{wap}) if for every $v\in V$ the $G$-orbit $G.v:=\{\rho_{g}(v)\mid g\in G\}\subseteq V$ is relatively weakly compact in $V$. \end{definition}

By the discussion in \cite{BaderRosendalSauer}, all linear representations on reflexive Banach spaces are weakly almost periodic. Another source of examples is given by the natural $L^{1}$-representations induced by actions via measure preserving transformations on probability spaces, see \cite[Example 10]{BaderRosendalSauer}.

As the following theorem demonstrates, weakly almost periodic representations fall within the setting of the previous sections.

\begin{theorem} \label{WeaklyAlmostPeriodicSetting} Let $G$ be a topological group and $(V,\rho)$ a wap weakly operator continuous Banach $G$-module. Then $(V,\rho)$ is BP-integrable and weakly uniquely $\mu$-stationary for every $\mu\in\mathcal{P}_{\sigma}(G)$. \end{theorem}

\begin{proof} It was argued in \cite{BaderRosendalSauer} that wap representations are uniformly bounded. Lemma \ref{MeasureDensity} implies that for every $v\in V$ and $\mu\in\mathcal{P}_{\sigma}(G)$ the operator $f_{v}\in V^{\ast\ast}$ given by $f_{v}(\phi):=\int_{G}\phi\left(\rho_{g}(v)\right)d\mu(g)$ is contained in the weak$^{\ast}$-closure of the convex hull of $\{\rho_{g}^{\ast\ast}(\widehat{v})\mid g\in G\}\subseteq\mathcal{B}(V^{\ast\ast})$. We can hence choose a net $(v_{\lambda})_{\lambda\in\Lambda}\subseteq\text{conv}(G.v)$ for which $\widehat{v_{\lambda}}\rightarrow f_{v}$ in the weak$^{\ast}$-topology. With the assumption that $(V,\rho)$ is weakly almost periodic, Krein's theorem (see e.g. \cite[Theorem 3.133]{FHHMZ}) gives that the closed convex hull of $G.v$ is weakly compact in $V$. By possibly going over to a subnet we can hence assume that $v_{\lambda}$ weakly converges to an element $w\in V$. But then $f_{v}(\phi)=\lim_{\lambda}\widehat{v_{\lambda}}(\phi)=\lim_{\lambda}\phi(v_{\lambda})=\phi(w)=\widehat{w}(\phi)$ for every $\phi\in V^{\ast}$ so that $f_{v}=\widehat{w}\in\text{im}(\iota_{V})$. It follows that $(V,\rho)$ is BP-integrable.

Let $\mathcal{C}$ be the norm closure of the convex hull of $G.v$. By Remark \ref{BPRemarks} (ii), $\rho_{\mu}^{i}(v)\in\mathcal{C}$ for every $i\in\mathbb{N}$. Since the operator $\mathbb{E}_{\mu}$ from Proposition \ref{Existence} is contained in the point-weak$^{\ast}$ closure of $\text{conv}\{(\rho_{\mu}^{\ast\ast})^{i}\mid i\in\mathbb{N}\}\subseteq\mathcal{B}(V^{\ast\ast})$, the image $\mathbb{E}_{\mu}(\widehat{v})$ is an element in the weak$^{\ast}$-closure of $\iota_{V}(\mathcal{C})\subseteq V^{\ast\ast}$. But this set coincides (by applying Krein's theorem once again) with the image of the weak closure of $\text{conv}(G.v)$ under $\iota_{V}$ and is therefore contained in $\text{im}(\iota_{V})$. It follows that $\mathbb{E}_{\mu}(\widehat{v})\in\text{im}(\iota_{V})$ for every $v\in V$ and hence, by Proposition \ref{EquivalentCharacterizations}, $(V,\rho)$ is weakly uniquely $\mu$-stationary. \end{proof}

In combination with \cite[Lemma 17]{BaderRosendalSauer}, Theorem \ref{FixedPointApproximation} and Theorem \ref{WeaklyAlmostPeriodicSetting} immediately imply the following corollary.

\begin{corollary} \label{WapConditional} Let $G$ be a topological group and $(V,\rho)$ a wap Banach $G$-module with $V$ being separable. Then $(V,\rho)$ is BP-integrable and there exists a Baire probability measure $\mu$ on $G$ with $\mathbb{E}_{\mu}(\widehat{v})\in\iota_{V}(V^{G})$ for every $v\in V$, where $\mathbb{E}_{\mu}$ is an operator as in Proposition \ref{Existence}. \end{corollary}

\vspace{3mm}



\section{Application to cohomology of groups\label{sec:Application-to-cohomology}}

In the following, we apply the results from the previous sections to the study of the reduced continuous cohomology of uniformly bounded representations.

\vspace{3mm}


\subsection{Complementation\label{subsec:Complementability}}

In \cite{Nowak17} Nowak initiated the study of the complementability of closed subspaces that appear in the cohomological framework beyond the setting of unitary representations on Hilbert spaces. For unitary representations results like this occur e.g. in Hodge theory, in the deduction of the vanishing of certain cohomology groups (see e.g. \cite{Garland73}), they were used in \cite{Ozawa18} to obtain a functional analytic proof of Gromov's polynomial growth theorem, and they are applied to rigidity questions for groups (see e.g. \cite{Shalom00}).

\begin{definition} A subspace $W$ of a topological vector space $V$ is \emph{complemented}, if there exists a subspace $W^{\prime}\subseteq V$ such that $V$ decomposes as a topological direct sum $V=W\oplus W^{\prime}$. \end{definition}

For linear representations on superreflexive Banach spaces and (more generally) weakly almost periodic representations, the complementation of $V^{G}$ in $V$ was studied in \cite{BaderFurmanGelanderMonod} and \cite{BaderRosendalSauer}. For $\mu$-integrable Banach modules, Proposition \ref{EquivalentCharacterizations} characterizes the complementability of the subspace of elements that are invariant with respect to the convolution with a given Baire probability measure.

The following statement is a refinement of the main theorem in \cite{Nowak17} (see also \cite[Example 6.1]{Rosendal22}).

\begin{proposition} \label{PiotrResult} Let $G$ be a topological group, $(V,\rho)$ a uniformly bounded strongly operator continuous Banach $G$-module, and $\mu\in\mathcal{P}_{\sigma}(G)$ a Baire probability measure. Assume that $\mu$ admits a null set $\mathcal{N}\subseteq G$ for which $G\setminus\mathcal{N}$ is contained in a compact set. If furthermore $\Vert\rho_{\mu}\Vert<1$, then $B^{1}(G,\rho)\subseteq Z^{1}(G,\rho)$ is a complemented closed subspace with

\[
Z^{1}(G,\rho)=B^{1}(G,\rho)\oplus\mathcal{H}_{\mu}^{1}(G,\rho),
\]
where $\mathcal{H}_{\mu}^{1}(G,\rho)$ is the subspace of all \emph{$\mu$-harmonic $1$-cocycles}, i.e. the subspace of all $b\in Z^{1}(G,\rho)$ with $\int_{G}b(h,e)d\mu(h)=0$. \end{proposition}

\begin{proof} First note that Pettis' measurability theorem (see Remark \ref{BPRemarks}) implies the well-definedness of the operator $\rho_{\mu}$. Furthermore, the function $h\mapsto b(h,e)$ is Bochner integrable for every 1-cocycle $b\in Z^{1}(G,\rho)$. By $\Vert\rho_{\mu}\Vert<1$ the operator $\Delta_{\mu}$ is (boundedly) invertible, so that $f_{b}\in C^{1}(G,V)^{G}$ given by 
\[
f_{b}(g):=\rho_{g}\circ\Delta_{\mu}^{-1}\left(\int_{G}b(h,e)d\mu(h)\right)
\]
is well-defined. Define a linear map $P:Z^{1}(G,\rho)\rightarrow B^{1}(G,\rho)$ via $P(b):=\partial^{1}(f_{b})$. One checks that $P$ is continuous with $P|_{B^{1}(G,\rho)}=\text{id}$. It follows that $B^{1}(G,\rho)\subseteq Z^{1}(G,\rho)$ is closed with $Z^{1}(G,\rho)=B^{1}(G,\rho)\oplus\ker(P)$. An element $b\in Z^{1}(G,\rho)$ is contained in the kernel of $P$ if and only if $\Delta_{\mu}^{-1}\left(\int_{G}b(h,e)d\mu(h)\right)\in V^{G}=\{0\}$, that is $b\in\mathcal{H}_{\mu}^{1}(G,\rho)$. \end{proof}

\begin{remark} \emph{(i)} Proposition \ref{PiotrResult} in particular applies to the case of isometric representations without almost invariant vectors on uniformly convex Banach spaces, see e.g. \cite[Lemma 5.2]{Rosendal22}.

\emph{(ii)} Note that the representation in Proposition \ref{PiotrResult} is weakly uniquely $\mu$-stationary with $\mathbb{E}_{\mu}\equiv0$.

\emph{(iii)} Every unitary representation of a locally compact group admits a decomposition as in Proposition \ref{PiotrResult}, see e.g. \cite{Guichardet72}, \cite{BekkaValette97}, \cite{Creutz11}, \cite{Creutz22}. \end{remark}

In combination with the ideas in \cite{BaderRosendalSauer} our previous results can be used to obtain complementation results for subspaces in higher degrees. Statements like this have earlier been studied by Rosendal in \cite{Rosendal22} in a more restrictive setting.

Motivated by an important ingredient in the proof of Shalom's cohomological rigidity results for unitary representations of locally compact groups in \cite{Shalom00}, in \cite{BaderRosendalSauer}, by applying the Ryll-Nardzewski fixed point Theorem, Bader, Rosendal and Sauer proved an analog of the following theorem in the setting of strongly operator continuous wap representations.

\begin{theorem} \label{MainTheorem} Let $G$ be a topological group containing subgroups $N,C\leq G$ with $C\subseteq C_{G}(N)$, where $C_{G}(N)$ denotes the \emph{centralizer} of $N$ in $G$. Let furthermore $\mu\in\mathcal{P}_{\sigma}(C)$ be a Baire probability measure, $(V,\rho)$ a uniformly bounded, strongly operator continuous, weakly uniquely $\mu$-stationary Banach $G$-module, and $b\in Z^{n}(G,\rho)$ a $n$-cocycle. Then $\widetilde{\mathbb{E}}_{\mu}\circ(b|_{N^{n+1}})$ defines an element in $Z^{n}(N,\rho|_{N})$ which is almost cohomologous to the restriction $b|_{N^{n+1}}$. \end{theorem}

\begin{proof} By $\mu\in\mathcal{P}_{\sigma}(C)$ the operator $\rho_{\mu}$ commutes with all elements $\rho_{m}\in\text{GL}(V)$, $m\in N$. From Proposition \ref{EquivalentCharacterizations} it follows that for all $m,m_{1},...,m_{n+1}\in N$, 
\begin{eqnarray*}
(\widetilde{\mathbb{E}}_{\mu}\circ b)(mm_{1},...,mm_{n+1}) & = & \lim_{N\rightarrow\infty}\frac{1}{N}\sum_{i=0}^{N-1}\rho_{\mu}^{i}\left(b(mm_{1},...,mm_{n+1})\right)\\
 & = & \lim_{N\rightarrow\infty}\rho_{m}\left(\frac{1}{N}\sum_{i=0}^{N-1}\rho_{\mu}^{i}\left(b(m_{1},...,m_{n+1})\right)\right)\\
 & = & \rho_{m}\left((\widetilde{\mathbb{E}}_{\mu}\circ b)(m_{1},...,m_{n+1})\right),
\end{eqnarray*}
i.e. $\widetilde{\mathbb{E}}_{\mu}\circ(b|_{N^{n+1}})\in C(N^{n+1},V)^{N}$. It is furthermore clear that $\partial^{n+1}(\widetilde{\mathbb{E}}_{\mu}\circ(b|_{N^{n+1}}))=0$, so that $\widetilde{\mathbb{E}}_{\mu}\circ(b|_{N^{n+1}})$ indeed defines an element in $Z^{n}(N,\rho|_{N})$.\\

Since $\text{im}(1-\widetilde{\mathbb{E}}_{\mu})=\overline{\text{im}(\Delta_{\mu})}$ is a $N$-invariant closed subspace of $V$, the element $b^{\prime}:=(1-\widetilde{\mathbb{E}}_{\mu})\circ(b|_{N^{n+1}})$ is a $n$-cocycle taking values in $\overline{\text{im}(\Delta_{\mu})}$. We may assume that $N$ is infinite. Denote by $\mathcal{S}$ the set of non-empty finite subsets of $N^{n+1}$, partially ordered by inclusion. For any $T\in\mathcal{S}$, Proposition \ref{ConvexHull} implies the existence of an element $\delta_{T}\in\text{conv}\{\rho_{c}\mid c\in C\}$ with $\max_{\mathbf{n}\in T}\Vert\widetilde{\mathbb{E}}_{\mu}(b(\mathbf{n}))-\delta_{T}(b(\mathbf{n}))\Vert<(\#T)^{-1}$. For $c\in C$ define a homomorphism $h_{c}:C^{n+1}(G,V)^{G}\rightarrow C^{n}(G,V)^{G}$ via 
\[
\left(h_{c}(f)\right)(g_{1},...,g_{n}):=\sum_{i=1}^{n}(-1)^{i+1}f(g_{1},...,g_{i},g_{i}c,...,g_{n}c).
\]
By the proof of \cite[Lemma 18]{BaderRosendalSauer}, $(1-\rho_{c})\circ(b|_{N^{n+1}})=\partial^{n}\left(h_{c}(b)|_{N^{n+1}}\right)$ and hence also $(1-\delta_{T})\circ(b|_{N^{n+1}})\in B^{n}(N,\rho|_{N})$. We claim that $(1-\delta_{T})\circ(b|_{N^{n+1}})\rightarrow b^{\prime}$ in $Z^{n}(N,\rho|_{N})$. Indeed, let $K\subseteq N^{n+1}$ be a norm compact set and let $\varepsilon>0$. By the continuity of $b$ we may find finitely many elements $\mathbf{m}_{1},...,\mathbf{m}_{k}\in K$ with $\frac{1}{k}<\frac{\varepsilon}{3}$ and $\min_{1\leq i\leq k}\Vert b(\mathbf{n})-b(\mathbf{m}_{i})\Vert<\frac{\varepsilon}{3M}$ for every $\mathbf{n}\in K$, where $M:=\sup_{g\in G}\Vert\rho_{g}\Vert$. For $T:=\{\mathbf{m}_{1},...,\mathbf{m}_{k}\}\in\mathcal{S}$ and $T\subseteq T^{\prime}\in\mathcal{S}$, 
\begin{eqnarray*}
 &  & \sup_{\mathbf{n}\in K}\Vert b^{\prime}(\mathbf{n})-(1-\delta_{T^{\prime}})\circ b(\mathbf{n})\Vert=\sup_{\mathbf{n}\in K}\Vert\widetilde{\mathbb{E}}_{\mu}(b(\mathbf{n}))-\delta_{T^{\prime}}(b(\mathbf{n}))\Vert\\
 & \leq & \frac{2\varepsilon}{3}+\max_{1\leq i\leq k}\Vert\widetilde{\mathbb{E}}_{\mu}(b(\mathbf{m}_{i}))-\delta_{T^{\prime}}(b(\mathbf{m}_{i}))\Vert<\frac{2\varepsilon}{3}+\frac{1}{\#T^{\prime}}<\varepsilon.
\end{eqnarray*}
This implies that 
\[
b^{\prime}=\lim_{T\in\mathcal{S}}(1-\delta_{T})\circ b\in\overline{B^{n}(N,\rho|_{N})},
\]
i.e. $\widetilde{\mathbb{E}}_{\mu}\circ(b|_{N^{n+1}})$ is almost cohomologous to $b|_{N^{n+1}}$ in $Z^{n}(N,\rho|_{N})$. \end{proof}

\begin{remark} By Theorem \ref{WeaklyAlmostPeriodicSetting} the statement in Theorem \ref{MainTheorem} in particular applies to wap representations and arbitrary measures. If $V$ is separable, Theorem \ref{FixedPointApproximation} furthermore implies in combination with \cite[Lemma 17]{BaderRosendalSauer} that $\mu$ can be chosen in a way such that $b|_{N^{n+1}}$ is almost cohomologous to a cocycle taking values in $V^{C}$. \end{remark}

\begin{corollary} \label{ComplementedCorollary} Let $G$ be a topological group containing subgroups $N,C\leq G$ with $N\subseteq C_{G}(N)$. Let furthermore $\mu\in\mathcal{P}_{\sigma}(C)$ be a Baire probability measure and $(V,\rho)$ a uniformly bounded, strongly operator continuous, weakly uniquely $\mu$-stationary Banach $G$-module. Then $\overline{B^{1}(N,V_{0})}$ with \emph{$V_{0}:=\overline{\text{im}(\Delta_{\mu})}$} topologically embeds as a complemented subspace into the image of the restriction map $Z^{1}(G,\rho)\rightarrow Z^{1}(N,\rho|_{N})$.

In particular, if $V^{\mu}\subseteq V^{N}$, then $\overline{B^{1}(N,\rho|_{N})}$ is complemented in the image of the restriction map $Z^{1}(G,\rho)\rightarrow Z^{1}(N,\rho|_{N})$. \end{corollary}

\begin{proof} Denote the restriction map $Z^{1}(G,\rho)\rightarrow Z^{1}(N,\rho|_{N})$ by $\text{Res}$ and note that $V_{0}$ is a $N$-invariant closed subspace of $V$. The linear map $C^{1}(N,V)^{N}\rightarrow C^{1}(G,V)^{G}$ given by $f\mapsto\overline{f}$ with $\overline{f}(g):=\rho_{g}\circ f(e)$ for $g\in G$ is well-defined, continuous and bijective. It induces a topological embedding of $\overline{B^{1}(N,\rho|_{N})}$ into $\text{Res}(Z^{1}(G,\rho))$ via $\partial^{1}f\mapsto\text{Res}(\partial^{1}\overline{f})$ for $f\in C^{1}(N,V)^{N}$. Similarly, $\overline{B^{1}(N,V_{0})}$ topologically embeds as a subspace of $\text{Res}(Z^{1}(G,\rho))$ via $\partial^{1}f\mapsto\text{Res}(\partial^{1}\overline{f})$ for $f\in C^{1}(N,V_{0})^{N}\subseteq C^{1}(N,V)^{N}$.

By Theorem \ref{MainTheorem} and Lemma \ref{Inclusion}, $(1-\widetilde{\mathbb{E}}_{\mu})\circ\text{Res}(b)\in\overline{B^{1}(N,V_{0})}$ for every $b\in Z^{1}(G,\rho)$. We thus obtain a map $P:\text{Res}(Z^{1}(G,\rho))\rightarrow\overline{B^{1}(N,V_{0})}$ which is well-defined and continuous. For every $f\in C^{1}(N,V_{0})^{N}$ and $m,n\in N$ we have that 
\[
P(\text{Res}(\partial^{1}\overline{f}))(m,n)=(1-\widetilde{\mathbb{E}}_{\mu})\left(f(n)-f(m)\right)=f(n)-f(m)=\text{Res}(\partial^{1}\overline{f})(m,n)
\]
so that $P$ restricts to the identity on $\overline{B^{1}(N,V_{0})}$. This implies that $\overline{B^{1}(N,V_{0})}$ is complemented in $\text{Res}(Z^{1}(G,\rho))$.\\

In the case where $V^{\mu}\subseteq V^{N}$, the space $\overline{B^{1}(N,V_{0})}$ identifies with $\overline{B^{1}(N,\rho|_{N})}$ in $\text{Res}(Z^{1}(G,\rho))$. Indeed, it is clear that $\overline{B^{1}(N,V_{0})}\subseteq\overline{B^{1}(N,\rho|_{N})}$. For the converse inclusion let $f\in C^{1}(N,V)^{N}$ be a function and note that by Proposition \ref{EquivalentCharacterizations} and the same argument as in the proof of Theorem \ref{MainTheorem}, $(1-\widetilde{\mathbb{E}}_{\mu})\circ f\in C^{1}(N,V_{0})^{N}$. For all $m,n\in N$ we have that $\widetilde{\mathbb{E}}_{\mu}\left(f(n)-f(m)\right)=0$ so that $\text{Res}(\partial^{1}\overline{f})=\text{Res}(\partial^{1}\overline{(1-\widetilde{\mathbb{E}}_{\mu})\circ f})$. From this, we obtain the converse inclusion and therefore, by the previous paragraph, that $\overline{B^{1}(N,\rho|_{N})}$ is complemented in $\text{Res}(Z^{1}(G,\rho))$. \end{proof}

If the subgroup is chosen to be the center of the group, the statement of Corollary \ref{ComplementedCorollary} extends to higher degrees.

\begin{corollary} \label{AbelianComplementedCorollary} Let $G$ be a topological group, $\mu\in\mathcal{P}_{\sigma}(\mathcal{Z}(G))$ a Baire probability measure, and $(V,\rho)$ a uniformly bounded, strongly operator continuous, weakly uniquely $\mu$-stationary Banach $G$-module. Then 
\begin{equation}
Z^{n}(G,\rho)=\overline{B^{n}(G,V_{0})}\oplus Z^{n}(G,V^{\mu}),\label{eq:Decomposition}
\end{equation}
where \emph{$V_{0}:=\overline{\text{im}(\Delta_{\mu})}$}. In particular, $\overline{H}_{c}^{n}(G,\rho)\cong\overline{H}_{c}^{n}(G,V^{\mu})$. \end{corollary}

\begin{proof} Theorem \ref{MainTheorem} implies that $b\mapsto\widetilde{\mathbb{E}}_{\mu}\circ b$ induces a well-defined linear map $P:Z^{n}(G,\rho)\rightarrow Z^{n}(G,V^{\mu})$ with $(1-P)(b)\in\overline{B^{n}(G,V)}$ for every $b\in Z^{n}(G,\rho)$. The map is obviously a continuous linear projection with $\text{im}(P)=Z^{n}(G,V^{\mu})$. By Lemma \ref{Inclusion} we furthermore have that $\ker(P)=\text{im}(1-P)=\overline{B^{n}(G,V_{0})}$, from which the decomposition (\ref{eq:Decomposition}) follows. \end{proof}

An analog to Corollary \ref{AbelianComplementedCorollary} for strongly operator continuous wap representations occurs in \cite[Corollary 5]{BaderRosendalSauer}. The statement notably generalizes a number of previous results, see \cite{Puls03}, \cite{MartinValette07}, \cite{Kappos05}, \cite{Nowak13}.

\vspace{3mm}


\subsection{Cohomology of nilpotent groups}

As before, given a Baire probability measure $\mu$ on a topological group $G$, let $G_{\mu}$ be the closed subgroup of $G$ generated by the set $\Sigma_{\mu}(G)$ of elements $g\in G$ for which $\int_{G}fd\mu>0$ for every positive function $f\in C_{b}(G,\mathbb{R})$ with $f(g)\neq0$. By combining the statement in Proposition \ref{CenterFixedPoint} with Corollary \ref{AbelianComplementedCorollary}, one can formulate an analog of \cite[Corollary 3.7]{Shalom00} and \cite[Theorem 7]{BaderRosendalSauer} for weakly uniquely stationary Banach modules. The proof is essentially the same as in \cite[Theorem 7]{BaderRosendalSauer}.

\begin{proposition} \label{CenterProposition} Let $G$ be a topological group, $\mu\in\mathcal{P}_{\sigma}(\mathcal{Z}(G))$ a Baire probability measure, and $(V,\rho)$ a uniformly bounded, strongly operator continuous, weakly uniquely $\mu$-stationary Banach $G$-module with $V^{G}=\{0\}$. Set $Z:=\mathcal{Z}(G)_{\mu}$ and assume that the product map $Z\times Z\rightarrow Z$ is measurable with respect to the product $\sigma$-algebra $\mathcal{B}a(Z)\otimes\mathcal{B}a(Z)$ and $\mathcal{B}a(Z)$. Then, $Z^{1}(G,V)\cong\overline{B^{1}(G,V_{0})}\oplus Z^{1}(G/Z,V^{Z})$, where \emph{$V_{0}:=\overline{\text{im}(\Delta_{\mu})}$}. In particular, $\overline{H}_{c}^{1}(G,\rho)\cong\overline{H}_{c}^{1}(G/Z,V^{Z})$. \end{proposition}

\begin{proof} Since $(V,\rho)$ is weakly uniquely $\mu$-stationary, we have by Corollary \ref{CenterFixedPoint} that $V^{\mu}=V^{Z}$. By Corollary \ref{AbelianComplementedCorollary} it thereby suffices to show that $Z^{1}(G/Z,V^{Z})\cong Z^{1}(G,V^{Z})$. One easily checks that $Z^{1}(G/Z,V^{Z})$ canonically embeds into $Z^{1}(G,V^{Z})$. For every $b\in Z^{1}(G,V^{Z})$ and $z\in Z$, $g,h\in G$ we furthermore have 
\[
b(g,zh)=b(g,e)+b(e,zh)=b(g,e)+b(e,h)-b(z,e)=b(g,e)+b(e,h)+b(e,z).
\]
For $g=h$ it in particular follows that $b(e,z)\in V^{G}=\{0\}$ and hence $b(g,zh)=b(zg,h)=b(g,h)$. This means that $b$ is constant on cosets of $Z$, so that $Z^{1}(G/Z,V^{Z})\cong Z^{1}(G,V^{Z})$. \end{proof}

Under similar conditions, we deduce the following implication.

\begin{proposition} \label{LimitBehaviour} Let $G$ be a topological group, $\mu\in\mathcal{P}_{\sigma}(\mathcal{Z}(G))$ a Baire probability measure, and $(V,\rho)$ a uniformly bounded, strongly operator continuous, weakly uniquely $\mu$-stationary Banach $G$-module. Set $Z:=\mathcal{Z}(G)_{\mu}$ and assume that the product map $Z\times Z\rightarrow Z$ is measurable with respect to the product $\sigma$-algebra $\mathcal{B}a(Z)\otimes\mathcal{B}a(Z)$ and $\mathcal{B}a(Z)$. Then, for any 1-cocycle $b\in Z^{1}(G,\rho)$ with $\lim_{n}n^{-1}\Vert b(e,z^{n})\Vert=0$ for all $z\in Z$, the element $\widetilde{\mathbb{E}}_{\mu}\circ b\in Z^{1}(G,V^{Z})$ factors through $G/Z$. \end{proposition}

\begin{proof} Since $(V,\rho)$ is weakly uniquely $\mu$-stationary we have by Corollary \ref{CenterFixedPoint} that $V^{\mu}=V^{Z}$. Theorem \ref{MainTheorem} therefore implies that $b^{\prime}:=\widetilde{\mathbb{E}}_{\mu}\circ b\in Z^{1}(G,V^{Z})$. By the same calculation as in the proof of Proposition \ref{CenterProposition}, $b(e,z^{n})=nb(e,z)$ for every $z\in Z$, $n\in\mathbb{N}$, so that by our assumption $\Vert b(e,z)\Vert=\lim_{n}n^{-1}\Vert b(e,z^{n})\Vert=0$. The claim then follows in the same way as in the proof of Proposition \ref{CenterProposition}. \end{proof}

\begin{theorem} \label{CommutatorDecomposition} Let $G$ be a finitely generated nilpotent topological group and $(V,\rho)$ a strongly operator continuous wap Banach $G$-module. Denote the Abelianization of $G$ by $G^{\text{ab}}:=G/[G,G]$. Then 
\begin{equation}
Z^{1}(G,\rho)\cong Z^{1}(G^{\text{ab}},V^{[G,G]})\oplus W\label{eq:NilpotentDecomposition}
\end{equation}
for some closed subspace $W\subseteq\overline{B^{1}(G,\rho)}$. The identification induces a topological isomorphism $\overline{H}_{c}^{1}(G,\rho)\cong\overline{H}_{c}^{1}(G^{\text{ab}},V^{[G,G]})$ and in the case where $V^{[G,G]}=V^{G}$ we have $W=\overline{B^{1}(G,\rho)}$. \end{theorem}

\begin{proof} Let $G$ be a finitely generated $c$-step nilpotent group and let $G=\gamma_{0}(G)\triangleright...\triangleright\gamma_{c}(G)=\{e\}$ be its lower central series. We perform an induction over $c$.

For $c=1$ the group $G$ is Abelian so that $[G,G]=\{e\}$ and $G^{\text{ab}}=G$. Therefore $Z^{1}(G,\rho)=Z^{1}(G^{\text{ab}},V^{[G,G]})$ and the claim immediately follows by choosing $W$ to be trivial.

For the induction step assume that for any finitely generated nilpotent topological group of step $<c$ and every strongly operator continuous wap Banach $G$-module the assertion of the theorem holds, and let $G$ be finitely generated $c$-step nilpotent. The subgroup $N:=\gamma_{c-1}(G)\leq\mathcal{Z}(G)$ is finitely generated, so that we may choose a finite symmetric generating set $S$ of $N$. Consider the corresponding Baire probability measure $\mu:=(\#S)^{-1}\sum_{s\in S}\delta_{s}$. By the main theorem in \cite{Shanbhag91} the pair $(N,\mu)$ is Liouville so that an application of Proposition \ref{LiouvilleFixedPoint}, Theorem \ref{WeaklyAlmostPeriodicSetting} and Corollary \ref{AbelianComplementedCorollary} give $Z^{1}(G,V)\cong\overline{B^{1}(G,V_{0})}\oplus Z^{1}(G,V^{N})$, where $V_{0}:=\overline{\text{im}(\Delta_{\mu})}$. Let $T$ be a finite generating set of $G$ and denote the corresponding word length function on $G$ by $\ell_{T}$. It is well-known (see e.g. \cite[Lemma 14.15]{Drutu18}) that every element $g\in N$ is distorted in the sense that $\lim_{n\rightarrow\infty}n^{-1}\ell_{T}(g^{n})=0$. In particular, 
\[
\lim_{n\rightarrow\infty}n^{-1}\left\Vert b(e,g^{n})\right\Vert \leq M\left(\max_{t\in T}\Vert b(e,t)\Vert\right)\left(\lim_{n\rightarrow\infty}n^{-1}\left\Vert \ell_{T}(g^{n})\right\Vert \right)=0
\]
for all $b\in Z^{1}(G,\rho)$, $g\in N$, where $M:=\sup_{g\in G}\Vert\rho_{g}\Vert$. The same reasoning as in the proof of Proposition \ref{LimitBehaviour} hence implies that for any 1-cocycle $b\in Z^{1}(G,\rho)$ the element $\widetilde{\mathbb{E}}_{\mu}\circ b\in Z^{1}(G,V^{N})$ factors through $G/N$. It therefore follows that 
\[
Z^{1}(G,\rho)\cong\overline{B^{1}(G,V_{0})}\oplus Z^{1}(G/N,V^{N}).
\]
The quotient $G/N$ is a finitely generated nilpotent topological group of step at most $c-1$ and the induced quotient representation of $G/N$ is strongly operator continuous and weakly almost periodic. Our induction assumption gives a decomposition 
\[
Z^{1}(G/N,V^{N})\cong Z^{1}\left((G/N)^{\text{ab}},V^{[G/N,G/N]}\right)\oplus W
\]
with a suitable closed subspace $W\subseteq\overline{B^{1}\left(G/N,V^{N}\right)}$. Note that the group $(G/N)^{\text{ab}}$ is a quotient of $G^{\text{ab}}$ and that the kernel of the quotient map acts trivially on $V^{[G,G]}$. This implies that $Z^{1}\left((G/N)^{\text{ab}},V^{[G/N,G/N]}\right)\cong Z^{1}(G^{\text{ab}},V^{[G,G]})$ canonically and hence 
\[
Z^{1}(G,\rho)\cong\left(\overline{B^{1}(G,V_{0})}\oplus W\right)\oplus Z^{1}(G^{\text{ab}},V^{[G,G]}).
\]
The space $\overline{B^{1}(G,V_{0})}\oplus W$ identifies with a subspace of $\overline{B^{1}(G,\rho)}$, from which we conclude the first part of the theorem. It furthermore immediately follows that $\overline{H}_{c}^{1}(G,\rho)\cong\overline{H}_{c}^{1}(G^{\text{ab}},V^{[G,G]})$.

Now assume that $V^{[G,G]}=V^{G}$ so that $Z^{1}(G^{\text{ab}},V^{[G,G]})\cong Z^{1}(G,V^{G})\subseteq Z^{1}(G,\rho)$. For $b\in Z^{1}(G,V^{G})$ and $g,h\in G$ we have 
\[
b(e,gh)=b(e,h)-b(g,e)=b(e,g)+b(e,h),
\]
which means that the map $g\mapsto b(e,g)$ is linear. An application of \cite[Corollary 3.3]{CTV07} gives that the map $g\mapsto b^{\prime}(e,g)$ is sublinear for every $b^{\prime}\in\overline{B^{1}(G,\rho)}$. Thereby the intersection of $Z^{1}(G,V^{G})$ and $\overline{B^{1}(G,\rho)}$ is trivial. With the decomposition (\ref{eq:NilpotentDecomposition}) we obtain the second part of the theorem. \end{proof}

\begin{remark} By similar arguments and an application of Theorem \ref{ErgodicTheorem} it can be shown that for every finitely generated 2-step nilpotent discrete group $G$, every uniformly bounded Banach $G$-module $(V,\rho)$, and every subgroup $H\leq[G,G]$ the image of the restriction map $\overline{H}_{c}^{1}(G,\rho)\rightarrow\overline{H}_{c}^{1}(H,\rho|_{H})$ is zero. Compare this with \cite[Theorem 8]{FernosValette12}. \end{remark}

\vspace{3mm}


\subsection{Cohomology of products of groups\label{subsec:Cohomology-of-products}}

The results in Subsection \ref{subsec:Complementability} indicate that the notion of weak unique stationarity is particularly useful in combination with commuting structures. Naturally, this motivates the study of the cohomology of products of groups. In this setting complementation results have been particularly important for the deduction of superrigidity statements for suitable unitary representations on Hilbert spaces and Banach modules; see for instance \cite{Shalom00}, \cite{BaderFurmanGelanderMonod} and also Subsection \ref{Subsec:InductionAndSuperrigidity}.

The following theorem is an application of Proposition \ref{PiotrResult}.

\begin{theorem} \label{NowakDirectProduct} Let $G:=G_{1}\times G_{2}$ be a product of topological groups and $(V,\rho)$ a uniformly bounded strongly operator continuous Banach $G$-module. Assume that $G_{1}$ and $G_{2}$ admit Baire probability measures $\mu_{1}\in\mathcal{P}_{\sigma}(G_{1})$, $\mu_{2}\in\mathcal{P}_{\sigma}(G_{2})$ and null sets $\mathcal{N}_{1}\subseteq G_{1}$, $\mathcal{N}_{2}\subseteq G_{2}$ for which both $G_{1}\setminus\mathcal{N}_{1}$ and $G_{2}\setminus\mathcal{N}_{2}$ are contained in compact sets. Assume furthermore that $\Vert\rho_{\mu_{1}}\rho_{\mu_{2}}\Vert<1$ and that $V^{\mu_{1}}=V^{G_{1}}$ and $V^{\mu_{2}}=V^{G_{2}}$. Then there exists a topological isomorphism $H_{c}^{1}(G_{1}\times G_{2},V)\cong H_{c}^{1}(G_{1},V^{G_{2}})\oplus H_{c}^{1}(G_{2},V^{G_{1}})$. \end{theorem}

\begin{proof} First recall that Pettis' measurability theorem implies the well-definedness of the operators $\rho_{\mu_{1}}$ and $\rho_{\mu_{2}}$, and that the functions $G_{1}\ni g\mapsto b(g,e)$ and $G_{2}\ni h\mapsto b(h,e)$ are Bochner integrable for every 1-cocycle $b\in Z^{1}(G,\rho)$. By the same argument as in the proof of Proposition \ref{PiotrResult}, we find that $Z^{1}(G,\rho)$ decomposes via $B^{1}(G,\rho)\oplus\mathcal{H}^{1}(G,\rho)$, where $\mathcal{H}^{1}(G,\rho)$ consists of all elements $b\in Z^{1}(G)$ with $\int_{G_{1}}(\int_{G_{2}}b(gh,e)d\mu_{2}(h))d\mu_{1}(g)=0$. This allows to restrict to the consideration of harmonic cocycles. Denote the corresponding projection onto $\mathcal{H}^{1}(G,\rho)$ by $P$.

For $b\in\mathcal{H}^{1}(G,\rho)$ set $v:=\int_{G_{1}}b(g,e)d\mu_{1}(g)$ and $w:=\int_{G_{2}}b(h,e)d\mu_{2}(h)$. We then have 
\[
\rho_{\mu_{1}}(w)+v=\int_{G_{1}}\left(\int_{G_{2}}b(gh,e)d\mu_{2}(h)\right)d\mu_{1}(g)=0
\]
and similarly $\rho_{\mu_{2}}(v)+w=0$, so that $\rho_{\mu_{1}}\rho_{\mu_{2}}(v)=\rho_{\mu_{1}}(-w)=v$. From this we obtain with $\Vert\rho_{\mu_{1}}\rho_{\mu_{2}}\Vert<1$ that $v=0$. In the same way, $w=0$. For $h\in G_{2}$, 
\begin{eqnarray*}
 &  & b(h,e)=\rho_{h}(v)+b(h,e)=\int_{G_{1}}b(gh,e)d\mu_{1}(g)\\
 & = & \int_{G_{1}}\rho_{g}\left(b(h,e)\right)d\mu_{1}(g)+v=\rho_{\mu_{1}}\left(b(h,e)\right),
\end{eqnarray*}
which means that $b(h,e)$ is invariant under $\rho_{\mu_1}$ and hence $b|_{G_{2}\times G_{2}}\in Z^{1}(G_{2},V^{G_{1}})$. A similar calculation leads to $b|_{G_{1}\times G_{1}}\in Z^{1}(G_{1},V^{G_{2}})$. We may therefore introduce a well-defined linear map $\pi:H_{c}^{1}(G,\rho)\rightarrow H_{c}^{1}(G_{1},V^{G_{2}})\oplus H_{c}^{1}(G_{2},V^{G_{1}})$ via $[b]\mapsto[Pb|_{G_{1}\times G_{1}}]\oplus[Pb|_{G_{2}\times G_{2}}]$. It remains to show that $\pi$ is an isomorphism of topological vector spaces.\\

\begin{itemize}
\item \emph{Injectivity}: Let $b\in\mathcal{H}^{1}(G,\rho)$ with $Pb|_{G_{1}\times G_{1}}\in B^{1}(G_{1},V^{G_{2}})$ and $Pb|_{G_{2}\times G_{2}}\in B^{1}(G_{2},V^{G_{1}})$. Then there exist functions $f_{1}\in C^{1}(G_{1},V^{G_{2}})^{G_{1}}$ and $f_{2}\in C^{1}(G_{2},V^{G_{1}})^{G_{2}}$ with $Pb|_{G_{1}\times G_{1}}=\partial^{1}f_{1}$ and $Pb|_{G_{2}\times G_{2}}=\partial^{1}f_{2}$. For $g,g^{\prime}\in G_{1}$, $h,h^{\prime}\in G_{2}$ we have 
\begin{eqnarray*}
b(gh,g^{\prime}h^{\prime}) & = & b(e_{G},g^{\prime}h^{\prime})-b(e_{G},gh)\\
 & = & \rho_{g^{\prime}}\left(b(e_{G},h^{\prime})\right)+b(e_{G},g^{\prime})-\rho_{g}\left(b(e_{G},h)\right)-b(e_{G},g)\\
 & = & b(e_{G},h^{\prime})+b(e_{G},g^{\prime})-b(e_{G},h)-b(e_{G},g)\\
 & = & \partial^{1}f_{2}(e,h^{\prime})+\partial^{1}f_{1}(e,g^{\prime})-\partial^{1}f_{2}(e,h)-\partial^{1}f_{1}(e,g)\\
 & = & f_{1}(g^{\prime})+f_{2}(h^{\prime})-f_{1}(g)-f_{2}(h).
\end{eqnarray*}
It follows that $b\in B^{1}(G,\rho)$ and hence that $\pi$ is injective.
\item \emph{Surjectivity}: For $b_{1}\in Z^{1}(G_{1},V^{G_{2}})$ and $b_{2}\in Z^{1}(G_{2},V^{G_{1}})$ define $b:G\times G\rightarrow V$ by $b(gh,g^{\prime}h^{\prime}):=b_{1}(g,g^{\prime})+b_{2}(h,h^{\prime})$, where $g,g^{\prime}\in G_{1}$, $h,h^{\prime}\in G_{2}$. For $g,g_{1}^{\prime},g_{2}^{\prime}\in G_{1}$, $h,h_{1}^{\prime},h_{2}^{\prime}\in G_{2}$ we have 
\begin{eqnarray*}
b(ghg_{1}^{\prime}h_{1}^{\prime},ghg_{2}^{\prime}h_{2}^{\prime}) & = & b_{1}(gg_{1}^{\prime},gg_{2}^{\prime})+b_{2}(hh_{1}^{\prime},hh_{2}^{\prime})\\
 & = & \rho_{g}\left(b_{1}(g_{1}^{\prime},g_{2}^{\prime})\right)+\rho_{h}\left(b_{2}(h_{1}^{\prime},h_{2}^{\prime})\right)\\
 & = & \rho_{gh}\left(b_{1}(g_{1}^{\prime},g_{2}^{\prime})+b_{2}(h_{1}^{\prime},h_{2}^{\prime})\right)\\
 & = & \rho_{gh}\left(b(g_{1}^{\prime}h_{1}^{\prime},g_{2}^{\prime}h_{2}^{\prime})\right),
\end{eqnarray*}
and 
\begin{eqnarray*}
\partial^{2}b(gh,g^{\prime}h^{\prime},g^{\prime\prime}h^{\prime\prime}) & = & b(g^{\prime}h^{\prime},g^{\prime\prime}h^{\prime\prime})-b(gh,g^{\prime\prime}h^{\prime\prime})+b(gh,g^{\prime}h^{\prime})\\
 & = & \left(b_{1}(g^{\prime},g^{\prime\prime})-b_{1}(g,g^{\prime\prime})+b_{1}(g,g^{\prime})\right)\\
 &  & +\left(b_{2}(h^{\prime},h^{\prime\prime})-b_{2}(h,h^{\prime\prime})+b_{2}(h,h^{\prime})\right)\\
 & = & 0,
\end{eqnarray*}
so that $b\in Z^{1}(G,\rho)$ with $[Pb|_{G_{1}\times G_{1}}]=[b_{1}]$, $[Pb|_{G_{2}\times G_{2}}]=[b_{2}]$. We obtain that $\pi$ is surjective. 
\item \emph{Homeomorphism}: Consider the map $Z^{1}(G,\rho)\rightarrow Z^{1}(G_{1},V^{G_{2}})\oplus Z^{1}(G_{2},V^{G_{1}})$ given by $b\mapsto Pb|_{G_{1}\times G_{1}} \oplus Pb|_{G_{2}\times G_{2}}$. This map is clearly continuous so that the induced map $Z^{1}(G,\rho)\rightarrow H_{c}^{1}(G_{1},V^{G_{2}})\oplus H_{c}^{1}(G_{2},V^{G_{1}})$ is also continuous. From the universal property of the quotient, it follows that $\pi$ must be continuous as well. The continuity of $\pi^{-1}$ follows in a similar way by considering the map $\eta:Z^{1}(G_{1},V^{G_{2}})\oplus Z^{1}(G_{2},V^{G_{1}})\rightarrow Z^{1}(G,\rho)$ given by $\eta(b_{1}\oplus b_{2})(gh,g^{\prime}h^{\prime}):=b_{1}(g,g^{\prime})+b_{2}(h,h^{\prime})$ for $b_{1}\in Z^{1}(G_{1},V^{G_{2}})$, $b_{2}\in Z^{1}(G_{2},V^{G_{1}})$, $g,g^{\prime}\in G_{1}$, $h,h^{\prime}\in G_{2}$. 
\end{itemize}
This finishes the proof. \end{proof}

\begin{remark} \label{DecompositionRemark1} Under the assumptions of Theorem \ref{NowakDirectProduct}, its assertion can alternatively be formulated as follows: every 1-cocycle in $Z^{1}(G,\rho)$ is cohomologous to a sum of the form $b_{1}+b_{2}$ with $b_{1},b_{2}\in Z^{1}(G,\rho)$, where $b_{1}$ takes values in $V^{G_{2}}$ and where $b_{2}$ takes values in $V^{G_{1}}$. Furthermore, $b_{1}$ and $b_{2}$ factor to elements in $Z^{1}(G_{1},V^{G_{2}})$ and $Z^{1}(G_{2},V^{G_{1}})$. \end{remark}

\begin{definition} Let $G$ be a topological group. A uniformly bounded Banach $G$-module $(V,\rho)$ is said to have \emph{almost invariant unit vectors} if for all compact subsets $K\subseteq G$ and $\varepsilon>0$ there exists an element $v\in V$ with $\Vert v\Vert=1$ and $\Vert v-\rho_{g}(v)\Vert<\varepsilon$ for all $g\in K$. \end{definition}

Theorem \ref{NowakDirectProduct} can be viewed as a partial generalization of \cite[Theorem C]{BaderFurmanGelanderMonod}, as the following remark illustrates.

\begin{remark} \label{UniformlyConvexExample} By applying the results in \cite{DrutuNowak19}, the assertion of Theorem \ref{NowakDirectProduct} in particular applies to isometric representations of compactly generated topological groups without almost invariant unit vectors on uniformly convex Banach spaces. \end{remark}

\begin{theorem} \label{WeaklyProductDecomposition} Let $G:=G_{1}\times G_{2}$ be a product of topological groups, $\mu_{1}\in\mathcal{P}_{\sigma}(G_{1})$, $\mu_{2}\in\mathcal{P}_{\sigma}(G_{2})$ Baire probability measures and $(V,\rho)$ a uniformly bounded strongly operator continuous Banach $G$-module. Assume that $(V,\rho|_{G_{1}})$ is weakly uniquely $\mu_{1}$-stationary without almost invariant unit vectors, and that $(V,\rho|_{G_{2}})$ is weakly uniquely $\mu_{2}$-stationary. Then there exists an embedding of topological vector spaces $H_{c}^{1}(G_{1}\times G_{2},V)\hookrightarrow H_{c}^{1}(G_{1},V^{\mu_{2}})\oplus H_{c}^{1}(G_{2},V^{\mu_{1}})$ via $[b]\mapsto[\widetilde{\mathbb{E}}_{\mu_{2}}\circ(b|_{G_{1}})]\oplus[\widetilde{\mathbb{E}}_{\mu_{1}}\circ(b|_{G_{2}})]$.

If furthermore $V^{\mu_{1}}=V^{G_{1}}$ and $V^{\mu_{2}}=V^{G_{2}}$, then the embedding above is surjective, i.e. $H_{c}^{1}(G_{1}\times G_{2},V)\cong H_{c}^{1}(G_{1},V^{G_{2}})$. \end{theorem}

\begin{proof} Denote the map above by $\iota$. In the following, we demonstrate that $\iota$ is a well-defined linear homeomorphism onto its image.\\

\begin{itemize}
\item \emph{Well-defined}: By $\mu_{2}\in\mathcal{P}_{\sigma}(G_{2})$ the operator $\rho_{\mu_{2}}$ commutes with all elements $\rho_{g}$, $g\in G_{1}$. From Proposition \ref{EquivalentCharacterizations} it therefore follows that $\widetilde{\mathbb{E}}_{\mu_{2}}\circ(f|_{G_{1}})\in C^{1}(G_{1},V^{\mu_{2}})^{G_{1}}$ and similarly $\widetilde{\mathbb{E}}_{\mu_{1}}\circ(f|_{G_{2}})\in C^{1}(G_{2},V^{\mu_{1}})^{G_{2}}$ for every $f\in C^{1}(G,V)^{G}$. But then 
\begin{eqnarray*}
\widetilde{\mathbb{E}}_{\mu_{2}}\circ\left((\partial^{1}f)|_{G_{1}\times G_{1}}\right)=\lim_{\lambda\in\Lambda}\partial^{1}\left(\widetilde{\mathbb{E}}_{\mu_{2}}\circ(f|_{G_{1}})\right)\in B^{1}(G_{1},V^{\mu_{2}})
\end{eqnarray*}
and $\widetilde{\mathbb{E}}_{\mu_{1}}\circ\left((\partial^{1}f)|_{G_{2}\times G_{2}}\right)\in B^{1}(G_{2},V^{\mu_{1}})$. This implies that $\iota$ is a well-defined linear map.
\item \emph{Injectivity}: Let $b\in Z^{1}(G,\rho)$ be a 1-cocycle for which $\widetilde{\mathbb{E}}_{\mu_{2}}\circ(b|_{G_{1}\times G_{1}})\in B^{1}(G_{1},V^{\mu_{2}})$, $\widetilde{\mathbb{E}}_{\mu_{1}}\circ(b|_{G_{2}\times G_{2}})\in B^{1}(G_{2},V^{\mu_{1}})$ and let $f\in C^{1}(G,V^{\mu_{2}})^{G_{1}}$ with $\widetilde{\mathbb{E}}_{\mu_{2}}\circ(b|_{G_{1}\times G_{1}})=\partial^{1}(f|_{G_{1}})$. By replacing the 1-cocycle $b$ by $b-\partial^{1}f$ we may assume that $\widetilde{\mathbb{E}}_{\mu_{2}}\circ(b|_{G_{1}\times G_{1}})=0$. Now consider the set $\mathcal{S}$ of non-empty finite subsets of $G_{1}\times G_{1}\leq G\times G$, partially ordered by inclusion. By the proof of Theorem \ref{MainTheorem} we find a net $(\delta_{T})_{T\in\mathcal{S}}\subseteq\text{conv}\{\rho_{g}\mid g\in G_{2}\}$ with $(1-\delta_{T})\circ(b|_{G_{1}\times G_{1}})\in B^{1}(G_{1},\rho|_{G_{1}})$ for every $T\in\mathcal{S}$ and $(1-\delta_{T})\circ(b|_{G_{1}\times G_{1}})\rightarrow(1-\widetilde{\mathbb{E}}_{\mu_{2}})\circ(b|_{G_{1}\times G_{1}})$ in $Z^{1}(G_{1},\rho|_{G_{1}})$. For $T\in\mathcal{S}$ and $x\in G_{2}$ let $0\leq\delta_{T}(x)\leq1$ be the corresponding coefficient of $\delta_{T}$, so that $\delta_{T}=\sum_{x\in G_{2}}\delta_{T}(x)\rho_{x}$. Since $(V,\rho|_{G_{1}})$ has no almost invariant unit vectors, there exists a non-trivial compact subset $K\subseteq G_{1}$ and $\varepsilon>0$ such that $\sup_{g\in K}\Vert v-\rho_{g}(v)\Vert\geq\varepsilon\Vert v\Vert$ for all $v\in V$. Let $K^{\prime}\subseteq G_{2}$ be a compact subset. For $g\in K$, $h\in K^{\prime}$ we have 
\begin{eqnarray*}
b(e_{G},gh) & = & \rho_{h}\left((1-\widetilde{\mathbb{E}}_{\mu_{2}})\circ b(e_{G},g)+\widetilde{\mathbb{E}}_{\mu_{2}}\circ b(e_{G},g)\right)+b(e_{G},h)\\
 & = & \rho_{h}\left((1-\widetilde{\mathbb{E}}_{\mu_{2}})\circ b(e_{G},g)\right)+b(e_{G},h)
\end{eqnarray*}
so that 
\begin{eqnarray*}
\sup_{g\in K,h\in K^{\prime}}\left\Vert \sum_{x\in G_{2}}\delta_{T}(x)\left(\rho_{h}\circ b(e_{G},x)-\rho_{h}\circ b(e_{G},gx)\right)\right\Vert \rightarrow0
\end{eqnarray*}
and hence by subtracting, 
\[
\sup_{g\in K^{\prime},h\in K}\left\Vert (1-\rho_{g})\left(\sum_{x\in G_{2}}\delta_{T}(x)\left(b(e_{G},hx)-b(e_{G},x)\right)\right)\right\Vert \rightarrow0.
\]
The assumptions on $K$ then imply that the net of functions of the form $G_{2}\ni h\mapsto\sum_{x\in G_{2}}\delta_{T}(x)\left(b(e_{G},hx)-b(e_{G},x)\right)$ converges to 0 uniformly on compact subsets. Set $M:=\sup_{g\in G}\Vert\rho_{g}\Vert$. From the estimate 
\begin{eqnarray*}
 &  & \Vert b(e_{G},gh)-\sum_{x\in G_{2}}\delta_{T}(x)(1-\rho_{gh})\circ b(e_{G},x)\Vert\\
 & = & \Vert\rho_{h}\left((1-\widetilde{\mathbb{E}}_{\mu_{2}})\circ b(e_{G},g)\right)+b(e_{G},h)-\sum_{x\in G_{2}}\delta_{T}(x)(1-\rho_{gh})\circ b(e_{G},x)\Vert\\
 & = & \Vert\rho_{h}((1-\widetilde{\mathbb{E}}_{\mu_{2}})\circ b(e_{G},g)-\sum_{x\in G_{2}}\delta_{T}(x)(1-\rho_{g})\circ b(e_{G},x))\\
 &  & \qquad+\sum_{x\in G_{2}}\delta_{T}(x)\left(b(e_{G},hx)-b(e_{G},x)\right)\Vert\\
 & \leq & M\Vert(1-\widetilde{\mathbb{E}}_{\mu_{2}})\circ b(e_{G},g)-(1-\delta_{T})\circ b(e_{G},g)\Vert\\
 &  & \qquad+\Vert\sum_{x\in G_{2}}\delta_{T}(x)\left(b(e_{G},hx)-b(e_{G},x)\right)\Vert
\end{eqnarray*}
it then follows that $\partial^{1}\xi_{T}\rightarrow b$ in $Z^{1}(G,\rho)$, where $\xi_{T}\in C^{1}(G,V)^{G}$ is given by $\xi_{T}(g):=\sum_{x\in G_{2}}\delta_{T}(x)\rho_{g}\circ b(e_{G},x)$. Since $(V,\rho|_{G_{1}})$ admits no almost invariant unit vectors, this implies that $b\in B^{1}(G,\rho)$. We thereby obtain that the map $\iota$ is injective.
\item \emph{Continuity}: The continuity of the map $\iota$ and its inverse follow similar to the proof of Theorem \ref{NowakDirectProduct}. 
\end{itemize}
This implies the first part of the statement of the theorem. The surjectivity of $\iota$ in the case where $V^{\mu_{1}}=V^{G_{1}}$ and $V^{\mu_{2}}=V^{G_{2}}$ follows in the same way as in the proof of Theorem \ref{NowakDirectProduct}. \end{proof}

\begin{remark} \label{DecompositionRemark2} Similar to Remark \ref{DecompositionRemark1}, under the assumptions of the second part of Theorem \ref{WeaklyProductDecomposition}, its assertion can alternatively be formulated as follows: every 1-cocycle in $Z^{1}(G,\rho)$ is cohomologous to a cocycle $b\in Z^{1}(G,\rho)$ which takes values in $V^{G_{2}}$ and which factors to an element in $Z^{1}(G_{1},V^{G_{2}})$. \end{remark}

Theorem \ref{WeaklyProductDecomposition} has in combination with Theorem \ref{WapConditional} the following immediate consequence.

\begin{corollary} Let $G:=G_{1}\times G_{2}$ be a product of topological groups and $(V,\rho)$ a strongly operator continuous wap Banach $G$-module. Assume that $V$ is separable and that $(V,\rho|_{G_{1}})$ has no almost invariant unit vectors. Then there exists a topological isomorphism ${H}_{c}^{1}(G_{1}\times G_{2},V)\cong{H}_{c}^{1}(G_{1},V^{G_{2}})$. \end{corollary}

\vspace{3mm}


\subsection{Induction and superrigidity\label{Subsec:InductionAndSuperrigidity}}

In \cite{Shalom00} and \cite{BaderFurmanGelanderMonod} decompositions as in Subsection \ref{subsec:Cohomology-of-products} were used to obtain superrigidity theorems for the continuous cohomology of suitable unitary representations on Hilbert spaces and Banach modules. We finish this article by following this path. For this, we first remind the reader of basic notions surrounding the theory of lattices and induced representations.

\vspace{3mm}


\subsubsection{Lattices}

Let $G=G_{1}\times...\times G_{n}$ be a product of locally compact groups and $\Gamma\leq G$ a discrete subgroup. Fix a Haar measure $\nu$ on $G$, normalized so that the induced measure on $G/\Gamma$ has total mass one. Recall that a subset $\mathcal{D}\subseteq G$ is called a \emph{fundamental domain} for $\Gamma$ if $G=\mathcal{D}\Gamma$ with $\gamma\mathcal{D}\cap\gamma^{\prime}\mathcal{D}=\emptyset$ for all $\gamma,\gamma^{\prime}\in\Gamma$ with $\gamma\neq\gamma^{\prime}$. The discrete subgroup $\Gamma$ is called a \emph{lattice} if it admits a Borel fundamental domain $\mathcal{D}$ with $\nu(D)<\infty$. A lattice is \emph{irreducible} if the image of every projection $p_{i}:\Gamma\rightarrow G_{i}$ with $1\leq i\leq n$ has dense image in $G_{i}$.

Given a fundamental domain $\mathcal{D}$ of a lattice $\Gamma\leq G_{1}\times...\times G_{n}$, we can consider the associated \emph{cocycle} $\chi_{\mathcal{D}}:G\times\mathcal{D}\rightarrow\Gamma$ defined by $\chi_{\mathcal{D}}(g,x):=\gamma$, where $\gamma$ is the unique element in $\Gamma$ with $gx\gamma\in\mathcal{D}$. One checks that $\chi_{\mathcal{D}}(gh,x)=\chi_{\mathcal{D}}(h,x)\chi\left(g,hx\chi_{\mathcal{D}}(h,x)\right)$ for all $g,h\in G$, $x\in\mathcal{D}$.

\begin{definition}[{\cite[Definition 8.2]{BaderFurmanGelanderMonod}}] \label{p-integrableDefinition} For $p>0$ a lattice $\Gamma\leq G_{1}\times...\times G_{n}$ is called \emph{$p$-integrable} if one of the following two statements holds:
\begin{enumerate}
\item The lattice is \emph{cocompact} (i.e. the quotient $G/\Gamma$ is compact);
\item The lattice it is finitely generated and for some finite generating set $S\subseteq G$ there exists a Borel fundamental domain $\mathcal{D}\subseteq G$ with $\int_{\mathcal{D}}\ell_{S}^{p}\left(\chi_{\mathcal{D}}(g,x)\right)d\nu(x)<\infty$ for all $g\in G$. Here $\ell_{S}$ denotes the word length function associated with $S$.
\end{enumerate}
\end{definition}

\vspace{3mm}


\subsubsection{Induced representations\label{subsec:Induced-representations}}

Let $G=G_{1}\times...\times G_{n}$ be a product of locally compact second countable groups and $\Gamma\leq G$ a $p$-integrable lattice. As before, fix a Haar measure $\nu$ on $G$, normalized so that the induced measure on $G/\Gamma$ has total mass one, and denote the corresponding measure on $G/\Gamma$ by $\widetilde{\nu}$. For a Banach space $V$ and $1<p<\infty$ consider the Banach space $L^{p}(G/\Gamma,V)$ of (equivalence classes of) measurable functions $f$ on $G/\Gamma$ with coefficients in $V$ for which $\Vert f\Vert$ is $p$-integrable. Given a strongly operator continuous isometric linear representation $\rho:\Gamma\rightarrow\text{GL}(V)$ and a fundamental domain $\mathcal{D}$ for $\Gamma$, we may introduce the induced representation $\text{Ind}_{\Gamma}^{G}(\rho):G\rightarrow\text{GL}(L^{p}(G/\Gamma,V))$ by 
\[
\left(\text{Ind}_{\Gamma}^{G}(\rho)_{g}(\xi)\right)(x\Gamma):=\rho_{\chi_{\mathcal{D}}(g^{-1},x)}\left(\xi\left(g^{-1}x\chi_{\mathcal{D}}(g^{-1},x)\Gamma\right)\right)
\]
for $x\in\mathcal{D}$, $\xi\in L^{p}(G/\Gamma,V)$. Then $\text{Ind}_{\Gamma}^{G}(\rho)$ is again a strongly operator continuous linear isometric representation. For notational convenience we will usually write $\widetilde{\rho}:=\text{Ind}_{\Gamma}^{G}(\rho)$. For a 1-cocycle $b\in Z^{1}(\Gamma,\rho)$ we furthermore define its \emph{induced 1-cocycle} $\widetilde{b}\in Z^{1}(G,\widetilde{\rho})$ via $\widetilde{b}(g,h)(x\Gamma):=b\left(\chi_{\mathcal{D}}(g^{-1},x),\chi_{\mathcal{D}}(h^{-1},x)\right)$. The assignment $b\mapsto\widetilde{b}$ induces a topological isomorphism $H^{1}(\Gamma,\rho)\cong H^{1}(G,\widetilde{\rho})$. For more details see \cite{Shalom00} and \cite{BaderFurmanGelanderMonod}.

\vspace{3mm}


\subsubsection{Superrigidity}

In the same way that Theorem \ref{NowakDirectProduct} can be viewed as a partial generalization of \cite[Theorem C]{BaderFurmanGelanderMonod} (see Remark \ref{UniformlyConvexExample}), the following theorem can be viewed as a partial generalization of \cite[Theorem D]{BaderFurmanGelanderMonod}. Its proof is essentially the same as the one of \cite[Theorem 4.1]{Shalom00}. We sketch it for completeness. Under suitable technical assumptions, the statement extends to arbitrary products of groups.

\begin{theorem} \label{NowakSuperrigidity} Let $G:=G_{1}\times G_{2}$ be a product of locally compact second countable groups, $\Gamma\leq G$ an irreducible $p$-integrable lattice, and $(V,\rho)$ a strongly operator continuous isometric Banach $\Gamma$-module. Assume that the groups admit Baire probability measures $\mu_{1}\in\mathcal{P}_{\sigma}(G_{1})$, $\mu_{2}\in\mathcal{P}_{\sigma}(G_{2})$ and null sets $\mathcal{N}_{1}\subseteq G_{1}$, $\mathcal{N}_{2}\subseteq G_{2}$ for which $G_{1}\setminus\mathcal{N}_{1}$, $G_{2}\setminus\mathcal{N}_{2}$ are contained in compact sets, for which $\Vert\widetilde{\rho}_{\mu_{1}}\widetilde{\rho}_{\mu_{2}}\Vert<1$, and which satisfy $\widetilde{V}^{\mu_{1}}=\widetilde{V}^{G_{1}}$, $\widetilde{V}^{\mu_{2}}=\widetilde{V}^{G_{2}}$. Finally, let $b\in Z^{1}(\Gamma,\rho)$ be a 1-cocycle. Then there exists for every $i\in\{1,2\}$ a $\Gamma$-invariant closed subspace $V_{i}\subseteq V$ on which the $\Gamma$-representation extends to a strongly operator continuous isometric linear $G$-representation that factors through a representation of $G_{i}$. Furthermore, there exist 1-cocycles $b_{1}\in Z^{1}(G,V_{1})$, $b_{2}\in Z^{1}(G,V_{2})$ such that $b$ is cohomologous to $b_{1}|_{\Gamma \times \Gamma}+b_{2}|_{\Gamma \times \Gamma}$ in $Z^{1}(\Gamma,V)$ and such that $b_{i}$ only depends on $G_{i}$ for $i\in\{1,2\}$. \end{theorem}

\begin{proof} Let $\mathcal{D}$ be a Borel fundamental domain for $\Gamma$ with $e\in\mathcal{D}$, which is either compact or satisfies the second condition in Definition \ref{p-integrableDefinition}. Furthermore, fix a 1-cocycle $b\in Z^{1}(\Gamma,\rho)$, consider the induced representation $\widetilde{\rho}:=\text{Ind}_{\Gamma}^{G}(\rho):G\rightarrow\text{GL}(L^{p}(G/\Gamma,V))$ and let $\widetilde{b}\in Z^{1}(G,\widetilde{\rho})$ be the induced 1-cocycle. Theorem \ref{NowakDirectProduct} (or rather Remark \ref{DecompositionRemark1}) implies in combination with our assumptions that $\widetilde{b}$ is cohomologous to a sum $\widetilde{b}_{1}+\widetilde{b}_{2}$ with $\widetilde{b}_{1},\widetilde{b}_{2}\in Z^{1}(G,\widetilde{\rho})$, where $\widetilde{b}_{i}$ takes for every $i\in\{1,2\}$ values in $G_{i}^{\prime}:=\prod_{j\neq i}G_{j}$ and $\widetilde{b}_{i}$ factors to an element in $Z^{1}(G_{i},V^{G_{i}^{\prime}})$.

Now fix $i\in\{1,2\}$ and consider the 1-cocycle $\widetilde{b}_{i}\in Z^{1}(G,\widetilde{\rho})$. As discussed in Subsection \ref{subsec:Induced-representations}, $H^{1}(\Gamma,\rho)\cong H^{1}(G,\widetilde{\rho})$ so that we find $\beta_{i}\in Z^{1}(\Gamma,\rho)$ with $\widetilde{b}_{i}-\widetilde{\beta}_{i}\in B^{1}(G,\widetilde{\rho})$. In other words, there exists a vector $\xi\in L^{p}(G/\Gamma,V)$ with $\widetilde{b}_{i}(g,h)=\widetilde{\beta}_{i}(g,h)+\widetilde{\rho}_{h}(\xi)-\widetilde{\rho}_{g}(\xi)$ for all $g,h\in G$. Since $\widetilde{b}_{i}$ factors to an element in $Z^{1}(G_{i},V^{G_{i}^{\prime}})$, the 1-cocycle vanishes on $G_{i}^{\prime}$ so that $\widetilde{\beta}_{i}(g,h)=\widetilde{\rho}_{g}(\xi)-\widetilde{\rho}_{h}(\xi)$ for all $g,h\in G_{i}^{\prime}$. Write $\gamma\cdot v:=\rho_{\gamma}(v)+\beta_{i}(e,\gamma)$ for the affine $\Gamma$-action on $V$ associated with $\beta_{i}$ and consider the closed affine subspace $V_{i}^{\prime}$ of $V$ consisting of all elements $v\in V$ for which $\gamma_{j}\cdot v\rightarrow v$ for every sequence $(\gamma_{j})_{j\in\mathbb{N}}\subseteq\Gamma$ with $p_{i}(\gamma_{j})\rightarrow e$, where $p_{i}:G\rightarrow G_{i}$ denotes the projection onto the $i$-th component of $G$. By the same argument as in \cite[Proposition 4.4]{Shalom00} (which again relies on \cite[Theorem 6.3]{Margulis91} and \cite[Lemma 1]{BekkaLouvet97}), we find that $V_{i}^{\prime}$ is non-empty. Note also that $V_{i}^{\prime}$ is invariant under the affine $\Gamma$-action. We claim that the action continuously extends to an action of $G$ on $V_{i}^{\prime}$ that factors through $G_{i}$. Indeed, by the irreducibility of $\Gamma\leq G$ we may for every $g\in G$ choose a sequence $(\gamma_{j}(g))_{j\in\mathbb{N}}\subseteq\Gamma$ with $p_{i}\left(\gamma_{j}(g)\right)\rightarrow p_{i}(g)$. Since $(\gamma_{j}(g)\cdot v)_{j\in\mathbb{N}}\subseteq V_{i}^{\prime}$ is a Cauchy sequence, for every $v\in V_{i}^{\prime}$ the limit $g\cdot v:=\lim_{j\in\mathbb{N}}\gamma_{j}(g)\cdot v$ exists. By the above, it is furthermore independent of the choice of $(\gamma_{j}(g))_{j\in\mathbb{N}}\subseteq\Gamma$, it factors through $G_{i}$, and it is continuous.

Now choose an element $v_{i}\in V^{\prime}$ and consider the closed subspace $V_{i}:=V_{i}^{\prime}+v_{i}$ of $V$. For every $\gamma\in\Gamma$ we have that $\rho_{\gamma}(v)=\gamma\cdot(v-v_{i})-\gamma\cdot v_{i}$. It thereby follows in combination with the above that $\rho$ extends to a strongly operator continuous isometric linear representation of $G$ on $V_{i}^{\prime}$. By a similar argument, the 1-cocycle $(\gamma,\gamma^{\prime})\mapsto\beta_{i}(\gamma,\gamma^{\prime})+\rho_{\gamma^{\prime}}(v_{i})-\rho_{\gamma}(v_{i})$ extends to an element $b_{i}$ in $Z^{1}(G,V_{i})$ that factors through $G_{i}$.

By $H^{1}(\Gamma,\rho)\cong H^{1}(G,\widetilde{\rho})$, the sum $\beta_{1}+\beta_{2}$ is cohomologous to $b$, so that $b$ must also be cohomologous to $b_{1}|_{\Gamma \times \Gamma}+b_{2}|_{\Gamma \times \Gamma}$. This implies the claim. \end{proof}

Analogously, the argument in the proof of Theorem \ref{NowakSuperrigidity} implies with Remark \ref{DecompositionRemark2} the following theorem.

\begin{theorem} \label{WeaklyProductSuperrigidity} Let $G:=G_{1}\times G_{2}$ be a product of locally compact second countable groups, $\Gamma\leq G$ a $p$-integrable lattice, $\mu_{1}\in\mathcal{P}_{\sigma}(G_{1})$, $\mu_{2}\in\mathcal{P}_{\sigma}(G_{2})$ Baire probability measures, and $(V,\rho)$ a uniformly bounded strongly operator continuous Banach $\Gamma$-module. Assume that $(\widetilde{V},\widetilde{\rho}|_{G_{1}})$ is weakly uniquely $\mu_{1}$-stationary without almost invariant unit vectors, that $(\widetilde{V},\widetilde{\rho}|_{G_{2}})$ is weakly uniquely $\mu_{2}$-stationary, and that $\widetilde{V}^{\mu_{1}}=\widetilde{V}^{G_{1}},\widetilde{V}^{\mu_{2}}=\widetilde{V}^{G_{2}}$. Finally, let $b\in Z^{1}(\Gamma,\rho)$ be a 1-cocycle. Then there exists a $\Gamma$-invariant closed subspace $V_{0}\subseteq V$ on which the $\Gamma$-representation extends to a strongly operator continuous isometric linear $G$-representation that factors through a representation of $G_{1}$. Furthermore, there exists a 1-cocycles $b_{0}\in Z^{1}(G,V_{0})$ such that $b$ is cohomologous to $b_{0}|_{\Gamma \times \Gamma}$ in $Z^{1}(\Gamma,V)$ and such that $b_{0}$ only depends on $G_{1}$. \end{theorem}

\begin{proof} Let $\mathcal{D}$ be a Borel fundamental domain for $\Gamma$ with $e\in\mathcal{D}$, which is either compact or satisfies the second condition in Definition \ref{p-integrableDefinition}, and fix a 1-cocycle $b\in Z^{1}(\Gamma,\rho)$. Consider the induced representation $\widetilde{\rho}:=\text{Ind}_{\Gamma}^{G}(\rho):G\rightarrow\text{GL}(L^{p}(G/\Gamma,V))$ and let $\widetilde{b}\in Z^{1}(G,\widetilde{\rho})$ be the induced 1-cocycle. Theorem \ref{WeaklyProductDecomposition} (or rather Remark \ref{DecompositionRemark2}) implies that $\widetilde{b}$ is cohomologous to a sum $\widetilde{b}_{1}+\widetilde{b}_{2}$ with $\widetilde{b}_{1},\widetilde{b}_{2}\in Z^{1}(G,\widetilde{\rho})$, where $\widetilde{b}_{i}$ takes for every $i\in\{1,2\}$ values in $G_{i}^{\prime}:=\prod_{j\neq i}G_{j}$ and $\widetilde{b}_{i}$ factors to an element in $Z^{1}(G_{i},V^{G_{i}^{\prime}})$. 

From here, the proof follows the one of Theorem \ref{NowakSuperrigidity}. \end{proof}

\vspace{3mm}



\end{document}